\documentclass{article}

\usepackage{caption}
\usepackage{longtable}
\usepackage{cellspace}
\usepackage{booktabs}
\usepackage{arydshln}
\usepackage{amsfonts}
\usepackage{amssymb}
\usepackage{graphicx}
\usepackage{amsmath}
\usepackage{amsthm}
\usepackage{amsmath}
\usepackage{amsfonts}
\usepackage{amsthm}
\usepackage{mathabx}
\usepackage[colorlinks=false,linkcolor=blue,citecolor= red,bookmarksopen=true]{hyperref}
\usepackage{comment}
\usepackage{enumerate}
\usepackage{dsfont}

\DeclareMathOperator{\essup}{\rm essup}

\newcommand{\abs}[1]{\lvert#1\rvert}
\newcommand{\norm}[1]{\lVert#1\rVert}

\newcommand{\R}{\mathbb{R}} 
\newcommand{\NN}{\mathbb{N}}
\newcommand{\QW}{\mathbb{Q}}

\newcommand{\E}{\mathbb{E}}

\newcommand{\ind}{\mathbf{1}}
\newcommand{\Ind}{\mathds{1}}

\newcommand{\PW}{\mathbb{P}}

\newcommand{\N}{\mathbb{N}}

\newcommand{\Borel}{\mathbb{B}}
\newcommand{\Linfty}{L^{\infty}(\Omega, \Fcal, \PW)}
\newcommand{\LinftyG}{L^{\infty}(\Omega, \Gcal, \PW)}

\newcommand{\brv}{\mathcal{L}^{\infty}(\Omega,\mathcal{F})}
\newcommand{\brvG}{\mathcal{L}^{\infty}(\Omega,\mathcal{G})}

\newcommand{\hPhi}{\widehat{\Phi}}
\newcommand{\hg}{\widehat{g}}
\newcommand{\tg}{\widetilde{g}}

\newcommand{\Ep}[1]{\E_\PW\left[#1\right]}

\newcommand{\Fcal}{{\mathcal F}}
\newcommand{\Gcal}{{\mathcal G}}

\newcommand{\Lcal}{{\mathcal L}}

\newcommand{\Ncal}{{\mathcal N}}

\newcommand{\Scal}{{\mathcal S}}
\newcommand{\Xcal}{{\mathcal X}}

\newcommand{\Tcal}{{\mathcal T}}
\newcommand{\Ucal}{{\mathcal U}}

\newcommand{\cm}[1]{\mathfrak{m}\left(#1\right)}
\newcommand{\snorm}[1]{\left\|#1\right\|_{\infty}}
\newcommand{\ale}[1]{{\color{red}#1}}

\newtheorem{theorem}{Theorem}[section]
\newtheorem{proposition}[theorem]{Proposition}
\newtheorem{corollary}[theorem]{Corollary}
\newtheorem{lemma}[theorem]{Lemma}

\newtheorem{assumption}[theorem]{Assumption}
\newtheorem{remark}[theorem]{Remark}

\newtheorem{definition}[theorem]{Definition}

\newtheorem{notation}[theorem]{Notation}

\begin{document}

\title{On Conditional Chisini Means and Risk Measures}
\date{\today}
\author{Alessandro Doldi\thanks{%
Dipartimento di Matematica, Universit\`{a} degli Studi di Milano, Via
Saldini 50, 20133 Milano, Italy, $\,\,$\emph{alessandro.doldi@unimi.it}. }
\and Marco Maggis\thanks{%
Dipartimento di Matematica, Universit\`a degli Studi di Milano, Via Saldini
50, 20133 Milano, Italy, \emph{marco.maggis@unimi.it}. }}
\maketitle

\begin{abstract}
\noindent Given a real valued functional $T$ on the space of bounded random variables, we investigate the problem of the existence of a conditional version of nonlinear means. We follow a seminal idea by Chisini (1929), defining a \textit{mean} as the solution of a functional equation induced by $T$. We provide sufficient conditions which guarantee the existence of a (unique) solution of a system of infinitely many functional equations, which will provide the so called Conditional Chisini mean. We apply our findings in characterizing the scalarization of conditional Risk Measures, an essential tool originally adopted by Detlefsen and Scandolo (2005) to deduce the robust dual representation. 
\end{abstract}


\parindent=0em \noindent
\section{Introduction}

In Probability Theory and Statistics the concept of \textit{mean} as introduced by Chisini \cite{Ch29} can be expressed as follows: given a vector $(x_1,\ldots,x_n)\in\R^n$ and a non-decreasing (with respect to component-wise order) function $T:\R^n\to\R$ the \textit{mean} is the value $\mathfrak{m} \in\R$ which solves the functional equation
\[T(x_1,x_2,\ldots,x_n)=T(\mathfrak{m},\mathfrak{m}\ldots,\mathfrak{m}).\]
Simple examples can be obtained by choosing strictly increasing $G,F:\R\to\R$ and
\[T(x_1,x_2,\ldots,x_n)= F\left(\sum_{i=1}^n G(x_i)\right).\]
In this case the unique solution
\[\mathfrak{m}=G^{-1}\left(\frac{G(x_1)+\ldots+G(x_n)}{n}\right)\] 
is independent from the choice of $F$ and $\mathfrak{m}$ is usually referred as Generalized Mean (see \cite{MP93} for a review). 

\medskip

The notion of \textit{mean} is crucial as soon as we extend the view of the problem to infinite dimensional spaces of random variables both in the classical linear case (e.g. the expected value which defines the $L^1$ space of integrable random variables) and in the nonlinear case  (see \cite{DenisHuPeng11,DenisMartini06} among the others). For this reason it is natural to consider a measurable space $(\Omega,\Fcal)$ where $\Omega$ is the set of all possible events and $\Fcal$ is a $\sigma$-algebra. \\ Let $\mathcal{X}$ be a vector space of  measurable random variables $f:\Omega\to\R$, such that  for any $A\in \Fcal$ the indicator function $\ind_A$ belongs to $ \mathcal{X}$, and $f\ind_A\in \mathcal{X}$ whenever $f\in\mathcal{X}$. 

\begin{definition}\label{standard:CM}
Given a functional $T: \mathcal{X}\to \R$ the Chisini mean  of $f\in \mathcal{X}$ will be given by $\mathfrak{m}(f)\in \R$ being solution of the functional equation 
\[T(f)=T(\mathfrak{m}(f)\ind_{\Omega}).\]
\end{definition}



\medskip

In this paper we shall concentrate our analysis on the space of bounded $\Fcal$-measurable random variables. We work in a model independent framework, in that we do not assume the knowledge of a probability measure $\PW$ on $(\Omega,\Fcal)$ \textit{a priori}. Our aim is to provide sufficient conditions for the existence of the conditional Chisini mean. More precisely: given a functional $T:\mathcal{X}\to \R$ and a $\sigma$-algebra $\Gcal\subseteq \Fcal$, we look for a $\Gcal$-measurable bounded random variable $g$ such that
\begin{equation}\label{functional:equation}
    T(g\ind_A)=T(f\ind_A) \quad \forall A\in \Gcal. 
\end{equation}
Under suitable assumptions on $T$ the solution will be unique only up to irrelevant events for $T$ (rigorously defined later) and the class of solutions will represent the conditional Chisini mean, which we will denote by $\cm{f| \Gcal}$ (see Definition \ref{CCM}).   

Equation \eqref{functional:equation} is a natural extension of the standard definition of conditional expectation. For a given probability space $(\Omega,\Fcal,\PW)$ the simplest example of Chisini mean is obtained considering $T$ as the Lebesgue integral of $f\in \mathcal{X}$ under $\PW$, namely $T(f):= E_{\PW}[f]$ and in this trivial case $\cm{f}$ coincides with $E_{\PW}[f]$. Moreover if $\Gcal$ is a sub $\sigma$-algebra of $\Fcal$ and $E_{\PW}[f| \Gcal]$ denotes the conditional expectation of $f$ given $\Gcal$, any version $g\in E_{\PW}[f| \Gcal]$ solves the system of infinitely many equations \eqref{functional:equation} with $T(f)=E_{\PW}[f]$.

One of the most celebrated results related to our research is the Nagumo-de Finetti-Kolmogorov Theorem (\cite{Na31,deF31,Ko30}), which provides an integral characterization of generalized means on finitely supported distributions. In \cite[Lemma 5.2]{CVMM} the Nagumo-de Finetti-Kolmogorov Theorem is extended to functionals defined on bounded random variables. In particular \cite[Lemma 5.2]{CVMM} leads to an immediate solution to \eqref{functional:equation} for a large class of functionals $T$ of the form $\cm{f\middle| \Gcal}= U^{-1}E_{\PW}[U(f)| \Gcal]$  for an increasing $U:\R\to\R$. The proof of \cite[Lemma 5.2]{CVMM} relies on restrictive assumptions like $(\Omega,\Fcal,\PW)$ being a non atomic probability space and $T$ satisfying $\PW$-law invariance \footnote{A functional $T$ is $\PW$-law invariant if it depends only on the laws of random variables i.e. $\PW(f\leq z)=\PW(g\leq z)$ for any $z\in\R$ implies $T(f)=T(g)$}. Our analysis diverges significantly from the Nagumo-de Finetti-Kolmogorov integral representation as we shall drop completely the $\PW$-law invariancy of $T$ and relax significantly the other assumptions (for example we shall not need an atomless measure space).

In the realm of Nonlinear Expectations, another significant example is the case of $g$-expectations (see \cite{CoquetPeng02} for a definition via \eqref{functional:equation} and  \cite{Peng04} for an exhaustive review). Given a filtered probability space $(\Omega,\Fcal,\{\Fcal_t\}_{t\in[0,T]},\PW)$, a $g$-expectation is loosely speaking a nonlinear functional which associates to a random variable $f$ the value of the solution of a Backward Stochastic Differential Equation (with driver $g$) at time $t$, i.e. $f_t=\mathcal{E}_t(f)$. If the solution of the BSDE exists unique then $f_t$ solves again the functional equation 
$\mathcal{E}_0(f\ind_A)= \mathcal{E}_0(f_t\ind_A)$ for any $A\in\Fcal_t$. Coquet and Peng \cite{CoquetPeng02} pursue an axiomatic approach to determine those families of functionals $\{\mathcal{E}_t\}_{t\in [0,T]}$ which can be represented as g-expectations. Their results lead to a solution of the conditional Chisini mean problem depicted so far. Nevertheless, it is important to observe that differently from \cite{CoquetPeng02} and all the related literature our framework does not require a Brownian underlying structure which is necessary to develop the theory of BSDEs, and for this reason we can avoid technical hypotheses like the so-called $\mathcal{E}_{\mu}$-domination.  

\section{Notations and preliminaries} \label{preliminaries}
For any given $\sigma$-algebra $\Gcal\subseteq
\Fcal$ we denote by $\mathcal{L}(\Omega,\Gcal)$ the space of
$\Gcal$-measurable functions taking values in $\R$, which will be always endowed with the Borel $\sigma$-algebra $\Borel_{\R}$. We shall usually refer to elements $f\in \mathcal{L}(\Omega,\Gcal)$ as random variables and denote by $\mathcal{L}^{\infty}(\Omega,\Gcal)$ its
subspace collecting bounded elements i.e. $f\in \mathcal{L}(\Omega,\Gcal)$ such that $|f(\omega)|\leq k$ for every $\omega\in\Omega$ and some $k\geq 0$. On $\mathcal{L}(\Omega,\Gcal)$ and
$\mathcal{L}^{\infty}(\Omega,\Gcal)$ we shall consider the usual
pointwise order $f\leq g$ if and only if $f(\omega)\leq g(\omega)$
for every $\omega\in \Omega$. 
$\mathcal{L}^{\infty}(\Omega,\Gcal)$
endowed with the sup norm $\|\cdot\|_{\infty}$ becomes a Banach
lattice, where $\|f \|_{\infty}= \sup_{\omega\in
\Omega}|f(\omega)|$. By $\ind_A$, $A\in\Gcal$ we indicate
the element of $\mathcal{L}^{\infty}(\Omega,\Gcal)$ such that
$\ind_A(\omega)=1$ if $\omega\in A$ and $0$ otherwise. 
Finally we shall denote by $\Scal(\Gcal)$ the subspace of simple functions in $\brvG$.

Whenever a probability $\PW$ is given $(\Omega,\Fcal,\PW)$
becomes a measure space and, as usual, we shall say that a
probability $\widetilde{\PW}$ is dominated by $\PW$
($\widetilde{\PW}\ll \PW$) if $\PW(A)=0$ implies
$\widetilde{\PW}(A)=0$ for any $A\in \Fcal$. Similarly a probability
$\widetilde{\PW}$ is equivalent to $\PW$
($\widetilde{\PW}\sim \PW$) if $\PW\ll\widetilde{\PW}$ and
$\widetilde{\PW}\ll \PW$. A property holds $\PW$ almost
surely ($\PW$-a.s.), if the set  where it fails is measurable and has $0$
probability.
\\For any given $\sigma$-algebra
$\Gcal\subseteq \Fcal$ we shall denote with $L^{0}(\Omega,\mathcal{G},\mathbb{P})$ the space of equivalence classes of
$\mathcal{G}$ measurable random variables that are $\mathbb{P}$
almost surely equal and by $L^{\infty}(\Omega
,\mathcal{G},\mathbb{P})$ the subspace of ($\PW$ a.s.) bounded
random variables. Formally any $f\in \mathcal{L}(\Omega,\Gcal)$ will be a representative of the class $X:=[f]_{\PW}\in L^{0}(\Omega,\Gcal,\PW)$. For $X,Y\in\Linfty$ we write $X\leq Y\,\PW-$a.s. for the usual a.s. ordering. Moreover, the essential ($\mathbb{P}$ a.s.)
\textit{supremum} of an arbitrary family of random variables
$\{X_{\lambda}\}_{\lambda\in\Lambda}\subseteq L^{0}(\Omega
,\mathcal{G},\mathbb{P})$ will be simply denoted by $\essup_\PW
\{X_{\lambda }\mid \lambda\in\Lambda\}$, and similarly for the
essential \textit{infimum} (see \cite[Section A.5]{FS11} for
details). 

\medskip

We now state a simple result for the existence and uniqueness of Chisini means as in Definition \ref{standard:CM}. It is interesting to notice that existence and uniqueness depend only on the regularity properties of the restriction of the functional $T$ on constant random variables.   

\begin{proposition}\label{unconditional}
Assume that $T:\brv\to \R$ is such that $T(-\snorm{f})\leq T(f)\leq T(\snorm{f})$ and the function $\R\ni a\mapsto T(a\ind_{\Omega})$ is continuous and strictly increasing. Then for any $f \in\brv$ there exists a unique $\mathfrak{m}\in\R$ such that $T(f)=T(\mathfrak{m}\ind_{\Omega})$.
\end{proposition}
\begin{proof} Continuity of the restriction guarantees that the sets $\{a\in\R\mid T(a\ind_{\Omega})\geq T(f)\}$ and $\{a\in\R\mid T(a\ind_{\Omega})\leq T(f)\}$ are closed. Moreover
$T(-\snorm{f})\leq T(f)\leq T(\snorm{f})$ implies that the two sets are non empty. Finally monotonicity implies that their union is the real line $\R$. Therefore their intersection is non empty and reduced to a singleton due to strict monotonicity of $T$ on constant random variables. 

\end{proof}

\begin{definition} 
For a given $T:\brv\to \R$ and any $\sigma$-algebra $\Gcal\subseteq \Fcal$ we introduce the class of irrelevant (or null) events in $\Gcal$ for the functional $T$ as
\begin{equation}
    \label{def:ncal initial}
    \Ncal_\Gcal=\{N\in\Gcal\mid T(g_1+g_2\ind_N)=T(g_1) \; \forall  g_1,g_2\in\brvG\}.
\end{equation}
We now list the key properties $T$ might enjoy, which will play a central role in the achievement of our scope.
\begin{description} 
    
    \item[($\Gcal$-Mo)] $T$ is $\Gcal$-monotone if, for all $x,y\in\R$ with $x<y$, all $g\in\brvG$ and all $A\in\Gcal\setminus\Ncal_\Gcal$, we have $$T(x\ind_A+g\ind_{\Omega\setminus A})< T(y\ind_A+g\ind_{\Omega\setminus A});$$
    
    \item[($\Gcal$-QL)] $T$ is $\Gcal$-quasilinear if, given any $g_1,g_2\in\brvG$ and $A\in\Gcal$,  $T(g_1\ind_A+\bar g\ind_{\Omega\setminus A})\leq T(g_2\ind_A+\bar g\ind_{\Omega\setminus A})$ for \textbf{some} $\bar g\in\brvG$ implies $T(g_1\ind_A+g\ind_{\Omega\setminus A})\leq T(g_2\ind_A+g\ind_{\Omega\setminus A})$ for \textbf{all}  $g\in \brvG$\footnote{As an immediate consequence, under ($\Gcal$-QL) alone, and given $g_1,g_2\in\brvG$, if we have $T(g_1\ind_A+\bar g\ind_{\Omega\setminus A})= T(g_2\ind_A+\bar g\ind_{\Omega\setminus A})$ for some $\bar g\in\brvG$, then $T(g_1\ind_A+g\ind_{\Omega\setminus A})= T(g_2\ind_A+g\ind_{\Omega\setminus A})$ for all  $g\in \brvG$: it is indeed enough to interchange the roles of $g_1,g_2$ in exploiting the assumption. };
    
    \item[($\Gcal$-PC)] $T$ is $\Gcal$-pointwise continuous if, for every norm bounded sequence $\{g_n\}\subseteq \brvG$ such that $g_n(\omega)\to g(\omega)$ for all $\omega\in\Omega$, we have $\lim_n T(g_n)=T(g)$.

\end{description}
For a \textbf{fixed} $f\in\brv$, we also introduce the following properties:
\begin{description}
    \item[($\Gcal$-PS)] $T$ is $\Gcal$- pasting at $f$ if, given $A_1,A_2\in\Gcal$ with $A_1\cap A_2=\emptyset$ and $x_1,x_2\in\R$ satisfying $T(f\ind_{A_i})=T(x_i\ind_{A_i})$  for $i=1,2$, then $$T(f\ind_{A_1}+f\ind_{A_2})= T(x_1\ind_{A_1}+x_2\ind_{A_2}).$$
    \item[($\Gcal$-NB)] $T$ is $\Gcal$-norm bounded at $f$ if $T(-\snorm{f}\ind_A)\leq T(f\ind_A)\leq T(\snorm{f}\ind_A)$ for every $A\in\Gcal$.
\end{description}
\end{definition}


In the previous definition we need to stress the dependence on the $\sigma$-algebra $\Gcal$, as the latter can range from  $\Gcal=\{\emptyset,\Omega\}$ to $\Gcal=\Fcal$. Notice moreover that ($\Gcal$-Mo), ($\Gcal$-QL)  and ($\Gcal$-PC) are properties regarding only the restriction of $T$ to $\brvG$.

\begin{remark}[From unconditional to conditional Chisini means]\label{discussion:properties} It is important to observe that the properties ($\Gcal$-Mo), ($\Gcal$-PC), ($\Gcal$-NB) collapse to the assumptions used in Proposition \ref{unconditional} as soon as $\Gcal=\{\emptyset, \Omega\}$. On the other hand it is no surprise that no counterparts of ($\Gcal$-QL) and ($\Gcal$-PS) appear in the statement of Proposition \ref{unconditional}, as both properties are always trivially satisfied for $\Gcal=\{\emptyset, \Omega\}$.
\end{remark}

\begin{remark} The following pasting property can be checked from ($\Gcal$-QL): for $g_i,\hg_i\in\brvG$ and $A_i\in\Gcal$ for $i=1,\ldots,N$ where the sets are mutually disjoint, we have
$$T(g_i\ind_{A_i})\geq T(\hg_i\ind_{A_i}) \;\forall \, i \quad \Rightarrow\quad  T\left(\sum_ig_i\ind_{A_i}\right)\geq T\left(\sum_i\hg_i\ind_{A_i}\right)$$
\end{remark}

Providing a definition of irrelevant events $\Ncal_\Gcal$ \textit{a priori} of any property of the functional $T$ is of primary importance to introduce the notion of ($\Gcal$-Mo) for $T$. As the characterizing condition for $\Ncal_\Gcal$ is somehow cumbersome and difficult to verify, we now state an equivalent, easier-to-handle formulation: if $T$ satisfies ($\Gcal$-QL), $T(0)=0$, and
\begin{equation}
    \label{weak monot}
    g_1,g_2\in\brvG,g_1(\omega)\leq g_2(\omega)\,\forall\omega\in\Omega\Longrightarrow T(g_1)\leq T(g_2),
\end{equation} 
then for $A\in \Gcal$ we have: 
\begin{equation}
    \label{formulalternativeN}
\Ncal_\Gcal=\left\{A\in\Gcal\mid T(x\ind_A)= 0\text{ for all }x\in\R\right\}.  
\end{equation}
The claim  is Lemma \ref{lemma: aux1} in the appendix. Moreover, $T(0)=0$, ($\Gcal$-Mo), ($\Gcal$-QL) and ($\Gcal$-PC) together imply \eqref{weak monot} (see Lemma \eqref{remonmonot}).

\section{Statement of the main result}\label{main:results}

For the rest of the paper we will always assume without loss of generality\footnote{In fact a translation of the functional by defining $\tilde T(\cdot)=T(\cdot)-T(0)$ does not affect the solution to \eqref{functional:equation}} that $T:\mathcal{L}^{\infty}(\Omega,\Fcal)\to \R$ satisfies $T(0)=0$..  \\ We recall that we are interested to the following problem of finding sufficient conditions under which $\cm{f\middle|\Gcal}\neq \emptyset$ where $\cm{f\middle|\Gcal}$ is given by the following
\begin{definition}[Conditional Chisini mean]\label{CCM} 
Let $(\Omega,\Fcal)$ be a measurable space. Consider $T:\brv\rightarrow\R$ and $f\in\brv$. We shall call conditional Chisini mean the set
\begin{equation}
\label{def: chisiniequivclass}
    \cm{f\middle|\Gcal}:=\left\{g\in\brvG\mid T(f\ind_A)=T(g\ind_A)\,\,\forall\,A\in\Gcal\right\}.
\end{equation}
\end{definition}

The following theorem provides sufficient conditions for the existence of conditional Chisini means for a large class of functionals and a fixed $\sigma$-algebra $\Gcal\subset \Fcal$. We stress that ($\Gcal$-PS) is a reasonable requirement if we want to guarantee existence of conditional Chisini means, as it is necessary already for a sigma algebra $\Gcal$ generated by two elements.

\begin{theorem}\label{thm:existencecond NEW} Assume that for a $\sigma$-algebra $\Gcal\subset \Fcal$ the functional $$T:\mathcal{L}^{\infty}(\Omega,\Fcal)\to \R$$ satisfies the properties ($\Gcal$-Mo), ($\Gcal$-QL), and ($\Gcal$-PC).
Then for any $f\in\brv$ for which ($\Gcal$-PS) and ($\Gcal$-NB) hold at $f$, there exists $\hg\in \brvG$ such that $T(f\ind_A)=T(\hg\ind_A)$ for all $A\in\Gcal$, i.e. $\hg\in \cm{f\middle|\Gcal}$. Moreover, such a $\hg$ is essentially unique in that for any $\tg\in\cm{f\middle|\Gcal}$, we have $\{\hg\neq \tg\}\in\mathcal{N}_{\Gcal}$.
\end{theorem}

\begin{remark}
    It is possible to mimick the arguments in \cite[Lemmas 3.3 to 3.6]{CoquetPeng02}, to obtain similar properties in our framework, literally substituting \textquotedblleft $\PW-a.s.$\textquotedblright with \textquotedblleft outside a measurable set in $\Ncal_\Gcal$\textquotedblright, and \textquotedblleft$\PW(A)>0$\textquotedblright with \textquotedblleft$A\in \Gcal\setminus\Ncal_\Gcal$\textquotedblright. In this way, one guarantees  monotonicity, tower property and homogeneity with respect to indicators of $\Gcal$-measurable events for the conditional Chisini mean.
\end{remark}

\section{Scalarization of conditional Risk Measures}
\label{example:RM}

The theory of Risk Measures is established  in Mathematical Finance, intertwining  Convex Analysis and Probability. A major branch of this theory is concerned with the risk assessment in dynamic frameworks, making extensive use of the notion of conditional Risk Measures which we here briefly recall (see \cite[Chapter 11]{FS11} for a detailed overview).  
\begin{definition}
\label{def:cond RM}
A functional $\rho_\Gcal:\Linfty\rightarrow\LinftyG$ is called conditional convex Risk Measure, if: $(i)$ $\rho_\Gcal(X)\leq \rho_\Gcal(Y)$ $\PW$-a.s. whenever $Y\leq X$ $\PW$-a.s.; $(ii)$ $\rho_\Gcal(X+c)=\rho_\Gcal(X)-c$, $\PW$-a.s. for all $c\in \LinftyG$; $(iii)$ for every $X_1,X_2\in \Linfty$, $\Lambda\in \LinftyG,0\leq \Lambda\leq 1,$ $$\rho_\Gcal(\Lambda X_1+(1-\Lambda)X_2)\leq \Lambda \rho_\Gcal(X_1)+(1-\Lambda)\rho_\Gcal(X_2) \quad \PW\text{-a.s.}$$ 
\end{definition}
Notice that in the previous sections we used random variables as real valued measurable functions on $(\Omega,\Fcal)$. In this context a reference probability $\PW$ is necessarily fixed a priori, hence as customary we will necessitate equivalence classes under $\PW$-a.s. equality. 
We will adopt the following notation: for any $A\in \Fcal$ we shall denote by $\Ind_A=[\ind_A]_{\PW}$ the equivalence class generated in $\Linfty$ by the indicator function $\ind_A\in \brv$.

\medskip

Given a conditional convex Risk Measure $\rho_\Gcal:\Linfty\rightarrow\LinftyG$, its \textbf{scalarization}, defined as the (non conditional) Risk Measure 
\begin{equation}\label{rho0}
\rho_0:\Linfty\to \R\text{, } \quad\rho_0(X):=\Ep{\rho_\Gcal(X)},    
\end{equation} 
turns out to play a key role in the analysis of the dual representation of conditional Risk Measures (see e.g. \cite[Theorem 1]{DetlefsenScandolo05}).
The main findings of this paper will allow us to characterize all the functionals $\rho_0$ which can be represented as the scalarization of a conditional Risk Measure as in \eqref{rho0}.

\begin{theorem}\label{scal:RM:prop}
Fix a probability space $(\Omega,\Fcal,\PW)$. Let $\rho_0:\Linfty\rightarrow\R$ be given, with $\rho_0(0)=0$.
Suppose that
\begin{enumerate}
\item $\rho_0(X+Y)=\rho_0(X)-\Ep{Y}$ for $X\in \Linfty$, $Y\in\LinftyG$;
\item for every $X_1,X_2\in \Linfty$, $X_1\leq X_2$ $\PW$-a.s. implies $\rho_0(X)\geq \rho_0(Y)$;
\item $\rho_0$ satisfies ($\Gcal$-PS) at $X$ for every $X\in \Linfty$\footnote{Meaning that the well defined map $f\mapsto\rho_0([f]_{\PW})$, $f\in\brv$ satisfies the property at any $f\in\brv$.} ; 
\item for every $N\geq 1$, for every $0\leq \lambda_1,\dots,\lambda_N\leq 1$ and $A_1,\dots,A_N\in\Gcal$ partition of $\Omega$, setting $\Lambda=\sum_{j=1}^N\lambda_j\Ind_{A_j}$, we have for all $X_1,X_2\in \Linfty$ 
\small\begin{equation}
    \label{eq:for convexity}
    \rho_0\left(\Lambda X_1+\left(1-\Lambda\right)X_2\right)\leq \sum_{j=1}^N\left(\lambda_j \rho_0(X_1\Ind_{A_j})+(1-\lambda_j) \rho_0(X_2\Ind_{A_j})\right).
\end{equation}
\end{enumerate}
Then there exists a conditional Risk Measure $\rho_\Gcal:\Linfty\rightarrow\LinftyG$, with $\rho_\Gcal(0)=0$ $\PW-$a.s., such that  $\rho_0(X)=\Ep{\rho_\Gcal(X)}$ for every $X\in \Linfty$.

\medskip

\noindent\textbf{Conversely:} for every conditional convex Risk Measure the map $\rho_0$ defined by \eqref{rho0} satisfies properties in item 1 to item 4.
\end{theorem}

\section{Proofs of Section \ref{main:results}}

This section is devoted to the proof of the main results of this paper.  Our research is inspired to the seminal paper \cite{De60} (see Appendix \ref{debreu}) jointly with an approach proposed in \cite{WZ99} that will be adapted to our scope in the proof of the following Theorem \ref{thm: reprTwithmeasure}. The overall argument 
will be involved and structured in several steps in the remainder of this section.  Nevertheless, the solution to \eqref{functional:equation} for a finitely generated $\sigma$-algebra $\Gcal$ is simple and independent from the aforementioned papers (see to this regard Proposition \ref{prop:finite NEW}). 


\begin{assumption}
\label{ass:conditional}
 $(\Omega, \Fcal)$ is a given measurable space.  $\Gcal\subseteq \Fcal$ is a sub $\sigma-$algebra.  $T:\brv\rightarrow \R$ is a given functional such that it satisfies ($\Gcal$-Mo), ($\Gcal$-QL), ($\Gcal$-PC),  and $T(0)=0$.
\end{assumption}

\begin{assumption}
\label{ass:conditional:local}
 $(\Omega, \Fcal)$ is a given measurable space.  $\Gcal\subseteq \Fcal$ is a sub $\sigma-$algebra.  $T:\brv\rightarrow \R$ is a given functional. $T$ satisfies ($\Gcal$-PS) and ($\Gcal$-NB) at a fixed $f\in \brv$.
\end{assumption}




\subsection{A representation result}\label{representation:measure}

\begin{definition}
For a given  measurable space $(\Omega,\Fcal)$, a sub $\sigma$-algebra $\Gcal\subseteq\Fcal$ and a functional $T:\brv\rightarrow\R$, we denote by $\Pi(\Gcal)$ the class of finite partitions $\pi$ of $\Omega$ obtained adopting $\Gcal$-measurable sets (i.e. $\pi\subseteq \Gcal$) such that for at least three distinct $A_1,A_2,A_3\in \pi$ we have $T(\ind_{A_j})>0,j=1,2,3$.
\end{definition}

\begin{theorem}
\label{thm: reprTwithmeasure}
Suppose Assumption \ref{ass:conditional} is satisfied and $\Pi(\Gcal)\neq \emptyset$. Then there exists a functional $V:\Gcal\times\brvG\rightarrow\R$, $(A,g)\mapsto V_A(g)$, such that: 
\begin{enumerate}
    \item for every $g\in\brvG$ the map $A\mapsto V_A(g)$ is a signed measure on  $(\Omega,\Gcal)$ with $V_A(0)=0$ for every $ A\in\Gcal$. Moreover $A\mapsto V_A(\ind_{\Omega})=:\PW(A)$ is a probability measure on  $(\Omega,\Gcal)$ such that $\Ncal_\Gcal = \{A\in \Gcal \mid \PW(A)=0\}$; 
    \item for every $g\in\brvG, A\in\Gcal$ we have $V_A(g)=V_A(g\ind_A)$;
    \item given $g_1,g_2\in \brvG$ and $A\in\Gcal$  we have $T(g_1\ind_A)\leq T(g_2\ind_A)$ if and only if $V_A(g_1)\leq V_A(g_2)$;
    \item for every $A\in\Gcal$ the functional $g\mapsto V_A(g)$, defined on  $\brvG$, is ($\Gcal$-PC).
\end{enumerate}
\end{theorem}

\subsubsection{Proof of Theorem \ref{thm: reprTwithmeasure}} 
\label{proofthm reprTwithmeasure}
The present Section \ref{proofthm reprTwithmeasure} is entirely devoted to the proof of the previous theorem 
and therefore \textbf{without further mention in the statements and proofs of the technical Lemmas, we shall always work under the hypotheses of Theorem \ref{thm: reprTwithmeasure}, i.e. Assumption \ref{ass:conditional} holds true and $\Pi(\Gcal)\neq \emptyset$.}

\begin{lemma}
\label{lemma: aux2}
Suppose Assumption \ref{ass:conditional} is satisfied. Then 
\begin{eqnarray}
\label{NT with indicators}
\Ncal_\Gcal & = & \{N\in\Gcal\mid T(\ind_N)=0 \}.
\end{eqnarray}
\end{lemma}
\begin{proof}
Lemma \ref{remonmonot} shows that condition \eqref{weak monot} automatically holds true.
If $T(\ind_A)=0$ but $A \in\Gcal\setminus \Ncal_\Gcal$, we could consider $x\in (0,1)$ so that from ($\Gcal$-Mo), $0=T(0)=T(0\ind_A)< T(x\ind_A)<T(\ind_A)=0$.

\end{proof}

\begin{lemma}
\label{Step 1}
Fix $\pi\in\Pi(\Gcal)$, let $\sigma(\pi)$ be the $\sigma$-algebra generated by $\pi$ and consider the vector space $\Scal(\sigma(\pi))$ of $\sigma(\pi)$-measurable functions. 
There exist functions $V_A^{\pi}:\Scal(\sigma(\pi))\to\R, A\in\pi$ such that for $f,g\in\Scal(\sigma(\pi))$:
\begin{equation}
    \label{eq: order preserving on simples NEW}
    T(g_1)\geq T(g_2)\text{ if and only if } \sum_{A\in\pi} V^{\pi}_{A}(g_1) \geq \sum_{A\in\pi} V^{\pi}_{A}(g_2);
\end{equation}
\begin{equation}\label{localityNEW}
V^{\pi}_A(g) = V^{\pi}_A(g\ind_A) \text{ for any } A\in \pi \text{ and } g\in \Scal(\sigma(\pi)).     
\end{equation}
\end{lemma}
\begin{proof}
For any $\pi\in\Pi(\Gcal)$ let $\sigma(\pi)$ be the $\sigma$-algebra generated by $\pi$. We consider the restriction of $T$ to  $\Scal(\sigma(\pi))$. In particular, as any $g\in \Scal(\sigma(\pi))$ can be represented in the form $\sum_{A\in\pi}x_A \ind_A$ the functional $T^{\pi}:\Scal(\sigma(\pi))\to \R$ will be defined as $T^{\pi}(\sum_{A\in\pi}x_A \ind_A)=T(\sum_{A\in\pi}x_A \ind_A)$. $T^{\pi}$ induces an order on $\R^{d}$ where $d=|\pi|\geq 3$ is the cardinality of partition $\pi$, namely
\[(x_1,\ldots,x_d)\succeq^{\pi} (y_1,\ldots,y_d) \text{ if and only if } T^{\pi}\big(\sum_{i=1}^d x_i \ind_{A_i} \big) \geq T^{\pi}\big(\sum_{i=1}^d y_i \ind_{A_i}\big).\]
First we show that the sets $\{(x_1,\ldots,x_d)\in\R^d\mid (x_1,\ldots,x_d)\succeq^{\pi} (y_1,\ldots,y_d)\}$,  $\{(x_1,\ldots,x_d)\in\R^d\mid  (y_1,\ldots,y_d) \succeq^{\pi} (x_1,\ldots,x_d)\}$ are closed in $\R^d$ (for simplicity we can adopt the sup norm on $\R^d$). Let $\{x^n\}\subset \R^d$ be a sequence converging to $x\in\R^d$ then clearly $h_n=\sum_{i=1}^d x_i^n \ind_{A_i}$ is $\snorm{\cdot}$ bounded in $\brvG$ and converges pointwise to $h=\sum_{i=1}^d x_i \ind_{A_i}$ (in particular it converges with respect to $\snorm{\cdot}$).  Property ($\Gcal$-PC) guarantees that $\lim T^{\pi}(h_n) = \lim T(h_n) = T(h) = T^{\pi}(h)$. ($\Gcal$-QL)  implies that  $\succeq^{\pi}$ satisfies the Sure Thing Principle (see Definition \ref{Sure:Thing}). Since we are assuming that at least three distinct elements in $\pi$ do not belong to $\mathcal{N}_{\Gcal}$, more than two indexes are essential for $\succeq^{\pi}$. We can  apply Theorem \ref{Debreu} guaranteeing  the existence of $V_A^{\pi}:\Scal(\sigma(\pi))\to\R$ for any $A\in\pi$ such that
\begin{equation}
    \label{eq: order preserving on simples}
    (x_1,\ldots,x_d)\succeq^{\pi} (y_1,\ldots,y_d) \text{ if and only if } \sum_{A\in\pi} V^{\pi}_{A}(g_1) \geq \sum_{A\in\pi} V^{\pi}_{A}(g_2),
\end{equation}
where $g_1=\sum_{i=1}^d x_i\ind_{A_i}$, $g_2=\sum_{i=1}^d y_i\ind_{A_i}$ and 
\[V^{\pi}_{A_i}(g_1):= U_i(x_i) \text{ where } U_i \text{ is given by Theorem \ref{Debreu}.}\]
As an immediate consequence of the definition of $V^{\pi}_{A_i}$  we deduce \eqref{eq: order preserving on simples NEW} and \eqref{localityNEW}. 

\end{proof}
\begin{remark}
Let $\{V^{\pi}_A\}_{A\in\pi}$ be the family obtained obtained in Lemma \ref{Step 1}. We observe that $x\mapsto V^{\pi}_A(x\ind_A)$ is  strictly increasing, continuous for any $A\in\pi$ as a consequence of ($\Gcal$-Mo) and ($\Gcal$-PC). In fact if $x<y$ and $A\in\pi$ is such that $T(\ind_A)>0$, then by ($\Gcal$-Mo) we have $T(x\ind_A)<T(y\ind_A)$. If we had $V^\pi_A(x)\geq V^\pi_A(y)$,  \eqref{eq: order preserving on simples NEW} would yield $T(x\ind_A)\geq T(y\ind_A)$, a contradiction. Continuity of $V_A$ is a direct consequence of the fact that the $U_i, i=1,\dots,N$ in Theorem \ref{Debreu} are guaranteed to be continuous.
\end{remark}
\begin{remark}
\label{rem: fix slope and intercetta}
In Lemma \ref{Step 1}, for a fixed partition $\pi$ we found an alternative representation to $T$ for the order $\succeq^{\pi}$, namely by $\sum_{A\in\pi} V^{\pi}_{A}$. For the remainder of the proof it is important to ensure that such new representation is uniquely determined.
Since Theorem \ref{Debreu} provides uniqueness of the representation up to increasing affine transformations, 
we can choose $\{V^\pi_A(\cdot)\}_{A\in\pi}$ such that $V^\pi_A(0)=0$ for all $A\in\pi$ and $\sum_{A\in\pi}V^\pi_A(1)=1 $. 
\end{remark}

We now take care of consistency over refinements. Fix two partitions $\pi,\pi'$ of $\Omega$, where $\pi'$ is  a refinement of $\pi$, i.e. any $A\in\pi$ can be written as the union of some $B\in j(A)\subset \pi'$ where $j(A)$ is the set $\{B\in \pi'\mid B\subseteq A\}$.
\\ If  $\pi'\in\Pi(\Gcal)$, applying Lemma \ref{Step 1} for $d' =|\pi'|$ we can find $ V^{\pi'}_B:\Scal(\sigma(\pi'))\to\R$ for any $B\in\pi'$ such that \eqref{eq: order preserving on simples NEW} and \eqref{localityNEW} hold true.

\begin{lemma}\label{Wakker NEW} Let $\pi\in\Pi(\Gcal)$ be given, and let be $\pi'$ be a refinement of $\pi$ such that $A\in\Gcal$ for all $A\in\pi'$. Then $\pi'\in\Pi(\Gcal)$ and for any $g\in \Scal(\sigma(\pi))$ and $A\in\pi$ we have $V^{\pi}_A(g\ind_A)=\sum_{B\in j(A)} V^{\pi'}_B(g\ind_B)$.
\end{lemma}

\begin{proof} 
First, we note that if $T(\ind_A)>0$, $B_1,\dots,B_N\in\Gcal$ are mutually disjoint and satisfy $A=\bigcup_{n=1}^NB_n$ (i.e. they are a measurable partition of $A$), then $T(\ind_{B_n})>0$ for at least one index $n=1,\dots,N$, as an immediate consequence of Lemma \ref{rem: propertiesirrelevant}. We conclude that $\pi'\in\Pi(\Gcal)$.

We now make the following observation: consider for $i=1,\ldots,d$ two families of functions $g_i:\R\to \R$ and $f_i:\R\to \R$ such that
\begin{equation}\label{uguaglianza NEW}\sum_{i=1}^d g_i(x_i)=\sum_{i=1}^d f_i(x_i) \text{ for all } (x_1,\ldots,x_d) \in\R^d.
\end{equation}
Then $g_k(x_k)- f_k(x_k) $ is constant for any $k$, since   $g_k(x_k)-f_k(x_k)=\sum_{i\neq k}^d f_i(x_i)-\sum_{i\neq j} g_i(x_i)$ for arbitrary $x_i\in\R,$
whose right hand side does not depend on $x_k$.
If in addition we impose $g_1(0)=f_1(0)=\ldots =g_n(0)=f_n(0)=0$, then $g_i(y)=f_i(y) $  for any $i$ and $y\in\R$.
\medskip
We choose $\pi'\in\Pi(\Gcal)$ which is a refinement of $\pi$ and an arbitrary $g\in \Scal(\sigma(\pi))$. We aim at comparing $\sum_{A\in \pi} V^{\pi}_A(g)$ with respect to $\sum_{B\in \pi'} V^{\pi'}_B(g)$. By Theorem \ref{Debreu} we have that necessarily for $g=\sum_{A\in\pi} x_A\ind_A$ 
\[\sum_{A\in \pi} V^{\pi}_A(g\ind_A)=a+\sum_{B\in \pi'} c V^{\pi'}_B(g\ind_B),\] 
where $a,c\in \R$ with $c>0$.
\\ Since from Remark \ref{rem: fix slope and intercetta} we have chosen $0=\sum_{A\in \pi} V^{\pi}_A(0)=\sum_{B\in \pi'} V^{\pi'}_B(0)$ and $1=\sum_{A\in \pi} V^{\pi}_A(1)=\sum_{B\in \pi'} V^{\pi'}_B(1)$ then necessarily $a=0$ and $c=1$.
\\ We conclude applying the initial observation to
\[\sum_{B\in \pi'} V^{\pi'}_B(g\ind_B)= \sum_{A\in \pi}\sum_{B\in j(A)}  V^{\pi'}_B(x_A\ind_B)=\sum_{A\in \pi} V^{\pi}_A(x_A\ind_A)\]
and obtaining therefore
\[\sum_{B\in j(A)}  V^{\pi'}_B(x\ind_B)=V^{\pi}_A(x\ind_A)\quad \text{for all } x\in\R.\]

\end{proof}

\begin{notation}
Notice that as a consequence of Lemma \ref{Wakker NEW} we can omit the dependence on the partition $\pi\in\Pi(\Gcal)$ when referring to the functional $V_A$.
We recall that we denote by $\Scal(\Gcal)$, the space of all $\Gcal$-measurable simple functions.
\end{notation}

 \begin{definition}
 \label{def: VA on simple}
 For any $g\in \Scal(\Gcal)$  and $A\in\Gcal$ we define 

\begin{equation}
    \label{defVAonsimple NEW}
    V_A(g) = \sum_{B\in \pi} V_{B\cap A}(x_B), 
\end{equation}
where $\pi\in \Pi(\Gcal)$ is such that $g= \sum_{B\in \pi} x_B\ind_B$ for suitable $x_B\in\R, B\in\pi$.
 \end{definition}
 
\begin{lemma}
\label{lemma first properties VA}
The following hold  for a fixed $A\in\Gcal$:
\begin{enumerate}[(i)]
    \item\label{item:wellposed} Definition \ref{def: VA on simple} is well posed;
    \item\label{item:concidesonindicators}  $V_A(0)=0$, $V_\Omega(1)=1$ and \begin{equation}
\label{coincideonindicators NEW}
  V_A(g)=V_A(g\ind_A)  \,\,\,\forall g\in\Scal(\Gcal);
\end{equation}
\item\label{item:VAnondecre} for $g_1,g_2\in\Scal(\Gcal)$ such that $g_1(\omega)\leq g_2 (\omega)\,\forall\omega\in\Omega$ we have $V_A(g_1)\leq V_A(g_2)$, and 
$V_A(g_1)< V_A(g_2)$ if $\bar A= \{g_1<g_2\}$ satisfies $T(\ind_{\bar A})>0$;
\item \label{item:TA orde pres}$V_A$ is $T(\cdot \ind_A)$-order preserving on $\Scal(\Gcal)$ (in the sense of Definition \ref{defop}).
\end{enumerate}
\end{lemma}

\begin{proof}
To check item \eqref{item:wellposed}, notice that any simple function can be written in the form $g= \sum_{B\in \pi} x_B\ind_B$ for some partition $\pi\in\Pi(\Gcal)$ if $\Pi(\Gcal)\neq \emptyset$, since as we showed in Lemma \ref{Wakker NEW} refinements (into $\Gcal$-measurable sets) of partitions in $\Pi(\Gcal)$ still belong to $\Pi(\Gcal)$. In particular, it is not necessary to impose $A\notin\Ncal_\Gcal$ for the well posedness of \eqref{defVAonsimple NEW}: indeed, if $\pi\in\Pi(\Gcal)$, then $\pi'=\{A\cap B\mid B\in\pi \}\cup\{(\Omega\setminus A)\cap B\mid B\in\pi\}$ is a refinement of $\pi$ satisfying $\pi'\subseteq\Gcal$, thus $\pi\in\Pi(\Gcal)$ ,\eqref{defVAonsimple NEW} makes sense, and it does not depend on $\pi$ from Lemma \ref{Step 1} and Lemma \ref{Wakker NEW}.

In item \eqref{item:concidesonindicators}, $V_A(0)=0, V_\Omega(1)=1$ are choices (see Remark \ref{rem: fix slope and intercetta} together with the definition of $V_A(\cdot) $ on simple functions in \eqref{defVAonsimple NEW}),  while \eqref{localityNEW} yields \eqref{coincideonindicators NEW}.

Item  \eqref{item:VAnondecre} is a direct consequence of the fact that we can represent $g_1,g_2$ adopting a common partition $\pi\in\Pi(\Gcal)$, i.e. $g_1=\sum_{B\in \pi} x_B\ind_B$ and $g_2=\sum_{B\in \pi} y_B\ind_B$. Moreover we proved in Lemma \ref{Step 1} that any $V_B(x)$ is strictly increasing in $x$ if $T(\ind_B)>0$.

Finally we  move to item \eqref{item:TA orde pres}. By construction $V_A$ is $T(\cdot \ind_A)$-order preserving on $\Scal(\Gcal)$, since whenever $g_1,g_2\in\Scal(\Gcal)$ are given, there exists a common $\pi\in\Pi(\Gcal)$ such that $A\in\sigma(\pi)$ and $g_1,g_2\in \Scal(\sigma(\pi))$. Also, by the definition \eqref{defVAonsimple NEW} and item \eqref{item:concidesonindicators}, $V_A(g_1)=V_A(g_1\ind_A)=\sum_{B\in\pi}V_{B\cap A}(g_1\ind_A)=\sum_{B\in\pi}V^\pi_{B}(g_1\ind_A)$ and the same holds for $g_2\ind_A$. By definition of $V^\pi_B(\cdot)$ (with $B\in\pi$) we have also that $T(g_1\ind_A)\leq T(g_2\ind_A)\Leftrightarrow T^\pi(g_1\ind_A)\leq T^\pi(g_2\ind_A)\Leftrightarrow \sum_{B\in\pi}V^\pi_{B}(g_1\ind_A)\leq \sum_{B\in\pi}V^\pi_{B}(g_2\ind_A)\Leftrightarrow V_A(g_1)\leq V_A(g_2)$.

\end{proof}
 
\medskip

We proceed by showing the following corollary which is an application of Lemma \ref{extension:brv NEW} in the Appendix. 

\begin{corollary}\label{extension NEW}
For any $A\in \Gcal$ the functional $V_A$ admits an extension to a $T(\cdot \ind_A)$-order preserving functional on $\brvG$, still denoted by $V_A$, which satisfies ($\Gcal$-PC),  $V_A(g)=V_A(g\ind_A)$ for every  $g\in\brvG$. 

\end{corollary}

\begin{proof}
By Lemma \ref{extension:brv NEW} for any $A\in \Gcal$, $V_A$ extends to a functional   $V_A: \brvG\rightarrow \R$ satisfying ($\Gcal$-PC). Given $g\in\brvG$ and $A\in\Gcal$ we can take a norm bounded  sequence of simple functions $g_n\in\Scal(\Gcal)$ such that  $g_n\downarrow_n g$ pointwise on $\Omega$. Then by \eqref{coincideonindicators NEW} we have $V_A(g_n)=V_A(g_n \ind_A)$. Since also $g_n\ind_A\downarrow_n g \ind_A$ we conclude by ($\Gcal$-PC) that  $V_A(g)=V_A(g\ind_A)$ for every  $g\in\brvG$.

\end{proof}

\begin{lemma}
\label{mufisameasure NEW}
For every $g\in\brvG$ and $A\in\Gcal$, the map $B\mapsto \mu_g^A(B):=V_{B\cap A}(g)$ defines a signed measure on $\Gcal$. 
\end{lemma}
\begin{proof}
We fix $A\in\Gcal$ throughout this proof. First we show that $B\mapsto V_{B\cap A}(g)$ is finitely additive. For $g\in\Scal(\Gcal)$, we can write $g=\sum_{B\in\pi}x_B\ind_B$ for some $\pi\in \Pi(\Gcal)$ using Lemma \ref{Wakker NEW}: it is enough to refine $\sigma(\pi)$ with an element of the nonempty set $\Pi(\Gcal)$. Given  $\{E_0,\dots,E_N\}\subset \Gcal$ a partition of $\Omega$, we also have $g=\sum_{B\in\pi}\sum_{j=1}^N x_B1_{B\cap E_j}$, and $\{B\cap E_j,B\in\pi,j=0,\dots,N\}\in\Pi(\Gcal)$ again by Lemma \ref{Wakker NEW}, so that 

\begin{align*}
\mu_g^A (\Omega)&=V_A(g)\stackrel{\odot}{=}\sum_{B\in\pi}\sum_{j=0}^NV_{A\cap  E_j\cap B}(x_B)=\sum_{j=0}^N\sum_{B\in\pi}V_{A\cap B\cap E_j}(x_B)\\&=\sum_{j=0}^N\sum_{B\in\pi}V_{A\cap E_j\cap  B}(g\ind_{B})\stackrel{\odot}{=}\sum_{j=0}^NV_{A\cap E_j}(g) =   \sum_{j=0}^N\mu_g^A(E_j)
\end{align*}
applying  the definition in \eqref{defVAonsimple NEW} in $(\odot)$.
Now, if $E_1,\dots,E_N$ are disjoint and $E_0=\Omega\setminus \bigcup_{j=1}^NE_j$ we get that from the previous computation (applied twice) 
\[\mu_g^A(E_0)+\mu_g^A\left(\bigcup_{j=1}^NE_j\right)=\mu_g^A (\Omega)=   \sum_{j=0}^N\mu_g^A(E_j)\]
which yields finite additivity for simple $g$. Take now $g\in\brvG$ and a norm bounded sequence of simple functions such that $g_n\rightarrow_n g$ pointwise on $\Omega$. Then for $\{E_1,\dots, E_N\}$ mutually disjoint and $E=\bigcup_{j=1}^N E_j$, applying ($\Gcal$-PC) obtained in Corollary \ref{extension NEW} $V_{A\cap E}(g)=\lim_n V_{A\cap E}(g_n)=\lim_n\sum_{j=1}^N V_{A\cap E_j}(g_n)=\sum_{j=1}^NV_{A\cap E_j}(g)$ implying finite additivity of $\mu_g^A$ for every $g\in \brvG$. Note also that $\mu_g^A(\emptyset):=V_{A\cap\emptyset}(g)=V_{A\cap\emptyset}(g\ind_{A\cap\emptyset})=V_{A\cap\emptyset}(0)=0$ for every $g\in\brvG$. Moreover 
\begin{equation}\label{consistency NEW}
V_A(g\ind_B)=V_{A\cap B}(g)\quad\forall\;g\in\brvG,\;A,B\in\Gcal
\end{equation}
just observing that
$V_A(g\ind_B)=V_{A\cap B}(g\ind_B)+V_{A\cap (\Omega\setminus B)}(g\ind_B)=V_{A\cap B}(g\ind_B)+0= V_{A\cap B}(g\ind_B)=V_{A\cap B}(g\ind_B\ind_{A\cap B})=V_{A\cap B}(g\ind_{A\cap B})=V_{A\cap B}(g)$.
Take now a sequence of disjoint events $(E_j)_j\subseteq \Gcal$. Noticing that $g\ind_{\bigcup_{j=1}^\infty E_j}=\lim_N g\ind_{\bigcup_{j=1}^N E_j}$ we conclude that
\begin{align*}
\mu_g^A\left(\bigcup_{j=1}^\infty E_j\right)&=V_{A\cap\bigcup_{j=1}^\infty E_j}(g)\stackrel{\eqref{consistency  NEW}}{=}V_A\left(g\ind_{\bigcup_{j=1}^\infty E_j}\right)
\\&\stackrel{\text{($\Gcal$-PC)}}{=}\lim_N V_A\left(g\ind_{\bigcup_{j=1}^N E_j}\right)=\lim_N\sum_{j=1}^NV_A(g\ind_{E_j})
\\&=\sum_{j=1}^\infty V_A(g\ind_{E_j})\stackrel{\eqref{consistency  NEW}}{=}\sum_{j=1}^\infty V_{A\cap E_j}(g)= \sum_{j=1}^\infty \mu_g^A(E_j)\,.
\end{align*}
Notice that for every $E\in\Gcal$ we have $\mu_{-\norm{g}_\infty}^A(E)\leq\mu_{g}^A(E) \leq \mu_{\norm{g}_\infty}^A(E)$ and both $-\mu_{-\norm{g}_\infty}^A(\cdot)$ and $ \mu_{\norm{g}_\infty}^A(\cdot)$ are (non negative) finite measures by the previous computations. Hence, the convergence of the series $\sum_{j=1}^\infty \mu_g^A(E_n)$ is absolute.

\end{proof}

The following Lemma sums up our findings, to provide the concluding details in proving Theorem \ref{thm: reprTwithmeasure}.
\begin{lemma}
Items 1 to 4 in Theorem \ref{thm: reprTwithmeasure} hold.
\end{lemma}
\begin{proof}
We check this item-by-item. Here, $(A,g)\mapsto V_A(g)$ is the one we constructed through Lemmas \ref{Step 1}-\ref{mufisameasure NEW}.
Items 2,3,4 are provided by Corollary \ref{extension NEW}.
As to item 1: from Corollary \ref{extension NEW}, $V_A(0)=0$ will clearly hold for the extension, since $0$ is a simple function. Lemma \ref{mufisameasure NEW} guarantees that $A\mapsto V_A(g)$ is a signed measure for every $g\in\brvG$, noticing that $V_{A}(g)=V_{\Omega\cap A}(g)$. Recall finally  that under the assumptions of Theorem \ref{thm: reprTwithmeasure}, the characterization \eqref{NT with indicators} holds and $A\in\Ncal_\Gcal$ if and only if $T(\ind_A)=0=T(0)$. 
    By Lemma \ref{remonmonot} we have $T(\ind_A)\geq T(0),\forall A\in\Gcal$, hence $V_A(\ind_{\Omega})\geq V_A(0)=0$ for all $A\in\Gcal$ by item 3 proved above. Moreover, Lemma \ref{lemma first properties VA} item (\ref{item:concidesonindicators}) yields $V_\Omega(\ind_{\Omega})=1$. Thus, $A\mapsto V_A(\ind_{\Omega})$ defines a probability measure.
    Since $V_A(\cdot)$ is $T(\cdot \ind_A)$-order preserving (item 3 above), we have that 
    $$A\in\Ncal_\Gcal\; \Leftrightarrow \; T(\ind_A)=T(0)=0  \; \Leftrightarrow \;\PW(A)=V_A(\ind_A)=V_A(0)=0.$$

\end{proof}

\subsection{Proof of Theorem \ref{thm:existencecond NEW}}

\begin{lemma}\label{lemma:pasting} Under Assumption \ref{ass:conditional} and Assumption \ref{ass:conditional:local} we have
\begin{enumerate}[(i)]
    \item for arbitrary finite collections of mutually disjoint sets $A_1,\ldots,A_N \in \Gcal$ and $g_1,\ldots, g_N \in \brvG$ such that $T(f\ind_{A_i})=T(g_i\ind_{A_i})$ for $i=1,\ldots,N$, 
    $$T\left(\sum_{n=1}^N f\ind_{A_i}\right)= T\left(\sum_{n=1}^N g_i\ind_{A_i}\right).$$
    Consequently the same holds replacing in both equations the equality with inequality;
    \item  $T(f\ind_{A}+f\ind_{N})=T(f\ind_A)$ for every $A\in\Gcal$ and $N\in\Ncal_{\Gcal}$ such that $A\cap N=\emptyset$.
\end{enumerate}
\end{lemma}

\begin{proof}
Item (i) is proved by induction. For the case $N=2$, take $x_1,x_2\in\R$ with $(T(f\ind_{A_i})=)T(x_i\ind_{A_i})=T(g_i\ind_{A_i})$. Such $x_1,x_2$ exist since $x\mapsto T(x\ind_A)$ is continuous (from ($\Gcal$-PC)) and $T(-\snorm{f}\ind_A)\leq T(f\ind_A)\leq T(\snorm{f}\ind_A)$
by ($\Gcal$-NB). By ($\Gcal$-Mo) there exists $x_A\in [-\snorm{f},\snorm{f}]$ such that $T(x_A\ind_A)=T(f\ind_A)$. Using ($\Gcal$-PS) we conclude that $T(f\ind_{A_1}+f\ind_{A_2})=T(x_1\ind_{A_1}+x_2\ind_{A_2})$. We now show that $T(x_1\ind_{A_1}+x_2\ind_{A_2})=T(g_1\ind_{A_1}+g_2\ind_{A_2})$:  by ($\Gcal$-QL) we have $T(x_1\ind_{A_1}+x_2\ind_{A_2})=T(x_1\ind_{A_1}+g_2\ind_{A_2})$ (from $T(x_2\ind_{A_2})=T(g_2\ind_{A_2})$, taking $A=A_2$, $\overline{g}=0$ and choosing $g=x_1\ind_{A_1}$)  and $T(x_1\ind_{A_1}+g_2\ind_{A_2})=T(g_1\ind_{A_1}+g_2\ind_{A_2})$ (from $T(x_1\ind_{A_1})=T(g_1\ind_{A_1}$, taking $A=A_1$, $\overline{g}=0$ and choosing $g=g_2\ind_{A_2}$).  The induction step from $N$ to $N+1$ goes as follows. First, consider $x_{N,N+1}\in\R$ with $T(f\ind_{A_{N}\cup A_{N+1}})=T(x_{N,N+1}\ind_{A_N\cup A_{N+1}})$, whose existence is guaranteed as above. Then 
$$T\left(\sum_{n=1}^{N-1} f\ind_{A_i}+f\ind_{A_N\cup A_{N+1}}\right)= T\left(\sum_{n=1}^{N-1} g_i\ind_{A_i}+x_{N,N+1}\ind_{A_N\cup A_{N+1}}\right).$$
Furthermore, it holds by the assumption on $g_N,g_{N+1}$, the definition of $x_{N,N+1}$ and the initial step of the induction that 
\begin{align*}
  &T(x_{N,N+1}\ind_{A_N\cup A_{N+1}})=T(f\ind_{A_N}+f\ind_{ A_{N+1}})\\
  =&T(g_N\ind_{A_N}+g_{N+1}\ind_{A_{N+1}})=T(\hg \ind_{A})  
\end{align*}
for $\hg=g_N\ind_{A_N}+g_{N+1}\ind_{A_{N+1}}$ and $A=A_N\cup A_{N+1}$. Invoking Assumption ($\Gcal$-QL), for $\overline{g}=\sum_{n=1}^{N-1} g_i\ind_{A_i}=\left(\sum_{n=1}^{N-1} g_i\ind_{A_i}\right)\ind_{\Omega\setminus A}$ we have 
$$T\left(\sum_{n=1}^{N-1} g_i\ind_{A_i}+x_{N,N+1}\ind_{A_N\cup A_{N+1}}\right)=T\left(\sum_{n=1}^{N-1} g_i\ind_{A_i}+\hg \ind_{A}\right)=T\left(\sum_{n=1}^{N} g_i\ind_{A_i}\right)$$
which completes the induction step recalling that the leftmost item equals $T\left(\sum_{n=1}^N f\ind_{A_i}\right)=T\left(\sum_{n=1}^{N-1} f\ind_{A_i}+f\ind_{A_N\cup A_{N+1}}\right)$.
\\ We now prove the case with inequalities. In particular let $A_1,\ldots,A_N \in \Gcal$ be a finite collections of mutually disjoint sets and $g_1,\ldots, g_N \in \brvG$ such that $T(f\ind_{A_i})\geq T(g_i\ind_{A_i})$ for $i=1,\ldots,N$. By ($\Gcal$-PC) and Lemma \ref{remonmonot} (in particular, since \eqref{weak monot} holds), for every $i=1,\ldots,N$ there exists $\varepsilon_i\geq 0$ such that $T(f\ind_{A_i}) = T((g_i+\varepsilon_i)\ind_{A_i})$. By ($\Gcal$-PS) and Lemma \ref{remonmonot} we conclude

    $$T\left(\sum_{n=1}^N f\ind_{A_i}\right)=T\left(\sum_{n=1}^N (g_i+\varepsilon_i)\ind_{A_i}\right)  \geq T\left(\sum_{n=1}^N g_i\ind_{A_i}\right).$$

We now come to item (ii). Take $x_A$ such that $T(f\ind_A)=T(x\ind_A)$ and observe that $T(0\ind_N)=0=T(-\norm{f}_\infty\ind_N)=T(\norm{f}\ind_N)$ since $\norm{f}\in\brvG$ and $N\in\Ncal_\Gcal$. Hence, using item (i) above, we get $T(f\ind_A+f\ind_N)=T(x_A\ind_A+0\ind_N)=T(x_A\ind_A)=T(f\ind_A)$.
\end{proof}

\begin{remark}
\label{rem: order cont at zero}
Observe that under Assumption \ref{ass:conditional},  for any sequence $(A_n)_n\subseteq\Gcal$ such that $A_n\downarrow_n \emptyset$ we have $T(f\ind_{A_n})\rightarrow_n T(0)$: indeed, by $(\Gcal$-NB), we have $T(-\snorm{f}\ind_{A_n})\leq T(f\ind_{A_n})\leq T(\snorm{f}\ind_{A_n}) $ and therefore we can invoke ($\Gcal$-PC) as $\pm\snorm{f}\in\brvG$.
\end{remark}
\begin{lemma}
\label{lemma:uniqueness}
Suppose that Assumption \ref{ass:conditional} is satisfied. Then if $\hg,\tg\in\cm{f\middle|\Gcal}$ we have  $\{\hg\neq \tg\}\in\mathcal{N}_\Gcal$.
\end{lemma}
\begin{proof}
Suppose $\{\hg>\tg\}\in\Gcal\setminus \Ncal_\Gcal$, then by Lemma \ref{remonmonot} we would have
$T(f\ind_{\{\hg>\tg\}})=T(\hg\ind_{\{\hg>\tg\}})>T(\tg\ind_{\{\hg>\tg\}})=T(f\ind_{\{\hg>\tg\}}), $ which yields a contradiction.
A similar argument provides $\{\hg<\tg\}\in\Ncal_\Gcal$, so that $\{\hg\neq\tg\}\in\Ncal_\Gcal$ by Lemma \ref{rem: propertiesirrelevant}.

\end{proof}

\begin{proposition}\label{prop:finite NEW}  Suppose Assumption \ref{ass:conditional} and Assumption \ref{ass:conditional:local} are satisfied. Take a partition $\pi$ of $\Omega$ with $\pi\subseteq \Gcal$. Then there exists a simple, $\sigma(\pi)-$measurable $\hg\in \Lcal^{\infty}(\Omega,\sigma(\pi))$ such that $T(f\ind_A)=T(\hg\ind_A)$ for any $A\in\Gcal$. Moreover, such a $\hg$ is essentially unique in that if $\tg\in\brvG$ satisfies the same properties, then  $\{\hg\neq \tg\}\in\mathcal{N}_\Gcal\cap\sigma(\pi)$.  
\end{proposition}
\begin{proof}
Observe that $x\mapsto T(x\ind_A)$ is continuous (from ($\Gcal$-PC)) and $$T(-\snorm{f}\ind_A)\leq T(f\ind_A)\leq T(\snorm{f}\ind_A)$$ 
by ($\Gcal$-NB). By ($\Gcal$-Mo), for every $A\in\pi$ there exists $x_A\in [-\snorm{f},\snorm{f}]$ such that $T(x_A\ind_A)=T(f\ind_A)$. For the fixed partition $\pi$ we set $\hg=\sum_{A\in\pi} x_{A}\ind_{A}\in \mathcal{L}^{\infty}(\Omega,\sigma(\pi))$. We conclude by noticing that any $B\in \sigma(\pi)$ can be written as $B=\cup_{j\in J} A_j$ where $J$ is an opportune choice of indexes. Therefore Lemma \ref{lemma:pasting} implies $T(f\ind_B)= T(g\ind_B)$.
\\ Essential uniqueness follows from Lemma \ref{lemma:uniqueness} observing that all the requirements in Assumption \ref{ass:conditional} hold if in particular we replace $\Gcal$ with $\sigma(\pi)\subseteq\Gcal$ when the latter appears in Assumption \ref{ass:conditional} itself. Hence, Lemma \ref{lemma:uniqueness} holds if we replace $\Gcal$ with $\sigma(\pi)$.

\end{proof}

\begin{lemma}
\label{lemma: thmfor pig empty}
Theorem \ref{thm:existencecond NEW}  holds in the case $\Pi(\Gcal)=\emptyset$. 
\end{lemma}
\begin{proof}
The case $\Omega\in\Ncal_\Gcal$ can be handled in a trivial way taking $\hg=0$ and noticing that $\hg\in\cm{f\middle|\Gcal}$. We then assume $\Omega\in\Gcal\setminus\Ncal_\Gcal$. This in particular implies that for any partition $\pi\subseteq \Gcal$, at least one $A\in\pi$ satisfies $A\in\Gcal\setminus \Ncal_\Gcal$ (otherwise, $\Omega$ would be the finite union of irrelevant events, thus irrelevant itself by Lemma \ref{rem: propertiesirrelevant}).
Notice that $\Pi(\Gcal)=\emptyset$ implies that for any partition $\pi\subseteq \Gcal$ of $\Omega$, at most two of its elements do not belong to $\Ncal_\Gcal$. This leaves us with two alternatives (up to shuffling the elements of the partition): $(i)$ for some partition $\pi=\{A_1,\dots,A_N\}$ of $\Omega$, with $\pi\subseteq \Gcal$, we have $T(\ind_{A_1})>0,T(\ind_{A_2})>0$; $(ii)$ for every partition  $\pi=\{A_1,\dots,A_N\}$ of $\Omega$, with $\pi\subseteq \Gcal$, we have $T(\ind_{A_1})>0$, and $T(\ind_{A_j})=0$ for every $j=2,\dots,N$.

First consider case $(i)$. Take $g\in\cm{f\middle|\sigma(\pi)}$, the latter being nonempty by Proposition \ref{prop:finite NEW}. Take $B\in\Gcal$, and take $\pi^B:=\{A\cap B\mid A\in\pi\}\cup\{A\cap (\Omega\setminus B)\mid A\in\pi\}\subseteq \Gcal$, which is again a partition of $\Omega$. By Lemma \ref{remonmonot} and $\Pi(\Gcal)=\emptyset$ we must have that at exactly two elements of $\pi^B$ do not belong to $\Ncal_\Gcal$, and that  one belongs to $\{A_1\cap B, A_1\cap (\Omega\setminus B)\}$, the other belongs to $\{A_2\cap B,A_2\cap (\Omega\setminus B) \}$. 
Suppose that $A_1\cap B\in\Gcal\setminus \Ncal_\Gcal$. We have that $T(g\ind_{A_1})=T(f\ind_{A_1})$ by construction, but also $T(f\ind_{A_1})=T(f\ind_{A_1\cap B}+f\ind_{A_1\cap (\Omega\setminus B)})=T(f\ind_{A_1\cap B})$ using Lemma \ref{lemma:pasting} item (ii), since $A_1\cap (\Omega\setminus B)\in\Ncal_\Gcal$, and analogously also $T(g\ind_{A_1})=T(g\ind_{A_1\cap B})$, so that $T(g\ind_{A_1\cap B})=T(f\ind_{A_1\cap B})$. If instead we had $A_1\cap B\in\Ncal_\Gcal$, we would still get $T(g\ind_{A_1\cap B})=T(f\ind_{A_1\cap B})$ much more easily. Similarly, we obtain $T(g\ind_{A_2\cap B})=T(f\ind_{A_2\cap B})$, in both cases $A_2\cap B\in\Gcal\setminus\Ncal_\Gcal$ and  $A_2\cap B\in\Ncal_\Gcal$. Now we apply ($\Gcal$-PS) and see that 
\begin{align*}
    T(g\ind_B)&\stackrel{}{=}T(g\ind_{A_1\cap B}+g\ind_{A_2\cap B}+g\ind_{(\Omega\setminus (A_1\cup A_2))\cap B})\stackrel{(\star)}{=}T(g\ind_{A_1\cap B}+g\ind_{A_2\cap B})\\
    &\stackrel{\text{($\Gcal$-PS)}}{=}  T(f\ind_{A_1\cap B}+f\ind_{A_2\cap B})\\
    &\stackrel{(\star\star)}{=}T(f\ind_{A_1\cap B}+f\ind_{A_2\cap B}+f\ind_{(\Omega\setminus (A_1\cup A_2))\cap B})=T(f\ind_B).
\end{align*}
where in $(\star),(\star\star)$ we used the fact that $(\Omega\setminus (A_1\cup A_2))\cap B$ is the finite union of elements of $\Ncal_\Gcal$ (since $\pi^B\notin \Pi(\Gcal)$).

We now  move to case $(ii)$: Take $\pi=\{\Omega\}$ and apply Proposition \ref{prop:finite NEW}, to get $g\in\cm{f\middle|\sigma(\pi)}$. Take now $B\in\Gcal$. We have that either $B\in\Ncal_\Gcal$, in which case $T(f\ind_B)=T(g\ind_B)$, or $B\in\Gcal\setminus \Ncal_\Gcal$. Since we are in case $(ii)$ we conclude in this case that $\Omega\setminus B\in\Ncal_\Gcal$. Thus, by \eqref{def:ncal initial}, $T(g\ind_\Omega)=T(g\ind_{\Omega\cap B})=T(g\ind_B)$ and by Lemma \ref{lemma:pasting} item (ii) $T(f\ind_\Omega)=T(f\ind_B)$. Finally, $T(g\ind_B)=T(f\ind_B)$, and this works for any $B\in\Gcal$, concluding the proof.

\end{proof}
\begin{remark}
\label{rem: equiv relation}
Suppose $\Pi(\Gcal)\neq \emptyset$. By Theorem \ref{thm: reprTwithmeasure} item 1 there exists a probability measure $\PW$ on $\Gcal$ describing exactly the null sets of $T$. Clearly, if $\QW\neq \PW$ is another probability measure on $\Gcal$ such that $\QW(A)=0\Leftrightarrow A\in\Ncal_\Gcal$, we must have $\PW\sim\QW$. Notice that the existence of more than one such measures is is not excluded by Theorem \ref{thm: reprTwithmeasure}.  
Lemma \ref{lemma:uniqueness} shows that $\cm{f\middle|\Gcal}$ (if nonempty) is an equivalence class of $\brvG$ for $\PW-$a.s. equality, for any $\PW$ probability measure on $\Gcal$ satisfying $\PW(A)=0\Leftrightarrow A\in\Ncal_\Gcal$. Equivalently, $\cm{f\middle|\Gcal}$ is an equivalence class in $\brvG$ for the equivalence relation $g_1\sim g_2\Leftrightarrow\{g_1\neq g_2\}\in\Ncal_\Gcal$. 
Observe that whenever $g_1\sim g_2$  we have $T(g_1\ind_A)=T(g_2\ind_A)$ for every $A\in\Gcal$. This follows directly from the definition of $\Ncal_\Gcal$ given in \eqref{def:ncal initial}. In particular then we also have $V_A(g_1)=V_A(g_2)\,\forall A\in\Gcal$ (see Theorem \ref{thm: reprTwithmeasure} item 3). We conclude that for any $A\in\Gcal$ we can actually induce a map (called again $V_A(\cdot)$) on the quotient space $\brvG/\sim$. Notice that the set $\brvG/\sim$ coincides with $L^\infty(\Omega,\Gcal,\QW)$ for every probability measure $\QW$ on $\Gcal$ such that for $A\in\Gcal$ we have $\QW(A)=0\Leftrightarrow A\in\Ncal_\Gcal$, and more precisely for every $g\in\brvG$ the equivalence classes of $g$ under $\sim$ and under $\QW-$a.s. equality are exactly the same.
\end{remark}

In the remainder of the section we write with an abuse of notation $V_A(\cm{f\middle|\Gcal})$, meaning $V_A(g)$ for some $g \in \cm{f\middle|\Gcal}$.

\begin{lemma}
\label{lemma:conditional ok on refinements}
Suppose Assumption \ref{ass:conditional} is satisfied and $\Pi(\Gcal)\neq\emptyset$. Fix $A\in \Gcal$. Take $\pi, \pi'\in\Pi(\Gcal)$ such that $A\in\sigma(\pi)\cap\sigma(\pi')$. Then \[V_A(\cm{f\middle|\sigma(\pi)})=V_A(\cm{f\middle|\sigma(\pi')}).\]
\end{lemma}
\begin{proof}
First, observe that by definition $$T(\cm{f\middle|\sigma(\pi)}\ind_A)=T(f\ind_A)=T(\cm{f\middle|\sigma(\pi')}\ind_A),$$ see Proposition \ref{prop:finite NEW}. Now, by Theorem \ref{thm: reprTwithmeasure} item 3 we conclude that the desired equality holds.

\end{proof}
By Lemma \ref{lemma:conditional ok on refinements} the following is well posed.
\begin{definition}
\label{def:VA[x]}
Suppose Assumption \ref{ass:conditional} is satisfied and $\Pi(\Gcal)\neq\emptyset$. 
For every $A\in\Gcal$ define $V_A[f]:=V_A(\cm{f\middle|\sigma(\pi)})$ for every $\pi\in\Pi(\Gcal)$ such that $A\in\sigma(\pi)$.
\end{definition}

\begin{proposition}
\label{prop: VA[f] is measure}
Suppose Assumption \ref{ass:conditional} is satisfied and $\Pi(\Gcal)\neq\emptyset$. Then
$A\mapsto V_A[f]$ defines a signed measure on $\Gcal$.
\end{proposition}

\begin{proof}
First we show finite additivity.  Let $\{E_1,\dots,E_N\}\subseteq\Gcal$ be a given partition of $E\in\Gcal$. Take $E_{0}:=\Omega\setminus E$ and $\pi=\{E_0,\dots,E_{N}\}$. Up to further refinements suppose $\pi\in\Pi(\Gcal)$ and $E_h\in\sigma(\pi)$ for every $h$ (this is possible by Lemma \ref{Wakker NEW}). Then $E\in\sigma(\pi)$ and $$V_E[f]=V_E(\cm{f\middle|\sigma(\pi)})\stackrel{\text{Thm.}\ref{thm: reprTwithmeasure}}{=}\sum_{j=1}^NV_{E_j}(\cm{f\middle|\sigma(\pi)})\,.$$
Observe now that $E_j\in\sigma(\pi)$ for every $j=1,\dots,N$, so that $V_{E_j}(\cm{f\middle|\sigma(\pi)})=V_{E_j}[f]$ by definition for every $j=1,\dots,N$.
We now show $\sigma$-additivity. Consider a sequence $(E_n)_n$ of disjoint sets in $\Gcal$, and set $\pi_N:=\{E_1,\dots,E_N,\bigcup_{k\geq N}E_k\}$ and $\pi\in\Pi(\Gcal)$ (which is nonempty by assumption). Denote by $\widehat{\pi}_N$ a common refinement between $\pi$ and $\pi_N$. Then $\widehat{\pi}_N\in \Pi(\Gcal)$ (Lemma \ref{Wakker NEW}) and $\bigcup_{k\geq N} E_k\in \sigma(\pi_N)\subset\sigma(\widehat{\pi}_N)$). Now, using the definition of $V_\cdot[f]$ and \eqref{consistency NEW}:
$$V_{\bigcup_{k\geq N}E_k}[f]=V_{\bigcup_{k\geq N}E_k}\left(\cm{f\middle|\sigma(\widehat{\pi}_N)}\right)=V_\Omega\left(\cm{f\middle|\sigma(\widehat{\pi}_N)}\ind_{\bigcup_{k\geq N}E_k}\right).$$
If we prove that right hand side tends to zero as $N$ increases, we are done.
Suppose by contradiction that for some subsequence, relabelled with the same index $N$, we had
$$\inf_N\abs{V_\Omega(\cm{f\middle|\sigma(\widehat{\pi}_N)}\ind_{\bigcup_{k\geq N}E_k})}>\varepsilon>0 \text{ for some } \varepsilon>0.$$ This can only happen in two cases, the first being: $$\inf_NV_\Omega(\cm{f\middle|\sigma(\widehat{\pi}_N)}\ind_{\bigcup_{k\geq N}E_k})>\varepsilon>0.$$ 
By pointwise continuity of $V_A(\cdot)$ (see Theorem \ref{thm: reprTwithmeasure} item 4) we have that there exists $x\in\R$ with $0=V_A(0)<V_\Omega(x)<\varepsilon$. Hence by Theorem  \ref{thm: reprTwithmeasure} item 3 we have $\inf_N T(\cm{f\middle|\sigma(\widehat{\pi}_N)}\ind_{\bigcup_{k\geq N}E_k})\geq T(x)>T(0)$. Also, we have by definition of $\cm{f\middle|\sigma(\widehat{\pi}_N)}$ that $T(\cm{f\middle|\sigma(\widehat{\pi}_N)}\ind_{\bigcup_{k\geq N}E_k})=T(f\ind_{\bigcup_{k\geq N}E_k})$, but then $\inf_N T(f\ind_{\bigcup_{k\geq N}E_k})\geq T(x)>T(0)$ would contradict Remark \ref{rem: order cont at zero}. The second case is $\sup_N V_\Omega(\cm{f\middle|\sigma(\widehat{\pi}_N)}\ind_{\bigcup_{k\geq N}E_k})<-\varepsilon<0$, which  can be handled similarly yielding again a contradiction.

\end{proof}

\begin{proposition}
\label{V[A]orderpres}
Suppose Assumption \ref{ass:conditional} is satisfied and $\Pi(\Gcal)\neq\emptyset$. Fix $g\in\brvG$ and $A\in\Gcal$. Then 

$$T(f\ind_A)\leq T(g\ind_A)\quad\Leftrightarrow \quad V_A[f]\leq V_A(g).$$

\end{proposition}
\begin{proof}
By Lemma \ref{Wakker NEW}, we can take $\pi\in\Pi(\Gcal)$ with $A\in\sigma(\pi)$.
We see that $T(f\ind_A)\leq T(g\ind_A)
   \Leftrightarrow T(\cm{f\middle|\sigma(\pi)}\ind_A)\leq T(g\ind_A)$, which follows from the definition of $\cm{f\middle|\sigma(\pi)}$ in Proposition \ref{prop:finite NEW}, and $ T(\cm{f\middle|\sigma(\pi)}\ind_A)\leq T(g\ind_A)
   \Leftrightarrow V_A(\cm{f\middle|\sigma(\pi)}\ind_A)\leq V_A(g\ind_A)$ is a consequence of item 2 in Theorem \ref{thm: reprTwithmeasure}. Additionally,  $ V_A(\cm{f\middle|\sigma(\pi)}\ind_A)\leq V_A(g\ind_A)
   \Leftrightarrow V_A[f]\leq V_A(g)$ follows from Definition \ref{def:VA[x]}.

\end{proof}

\begin{proposition}
\label{prop: increase locally}
Suppose Assumption \ref{ass:conditional} is satisfied and $\Pi(\Gcal)\neq\emptyset$. Let $g\in\brvG$ be given. Suppose that for some $B\in\Gcal$ we have $T(f\ind_B)>T(g\ind_B)$. Then there exist $\varepsilon>0$ and $\Omega_0\in\Gcal\setminus\Ncal_\Gcal$ such that $T(f\ind_A)\geq T\left((g+\varepsilon\ind_{\Omega_0}\right)\ind_A)$ for every $A\in\Gcal,A\subseteq \Omega_0$. 
\end{proposition}

\begin{proof}
We start observing that by ($\Gcal$-PC) we have  $T(f\ind_B)> T((g+\varepsilon)\ind_B)$ for some $\varepsilon>0$, which implies $ V_A[f]-V_A\left(g+\varepsilon\right)$ . Define now the signed measure $\mu$ as $A\mapsto \mu(A):= V_A[f]-V_A\left(g+\varepsilon\right)$. Observe that $\mu(B)>0$ by Proposition \ref{V[A]orderpres}. Take the Hahn Decomposition of $\mu$ (see \cite[Theorem 6.14]{Rudin87}), with  $\Omega_0\in\Gcal$ satisfying $\mu(A)\geq 0$ for every $A\in\Gcal,A\subseteq\Omega_0$ and $\mu(A)\leq 0$ for every $A\in\Gcal,A\subseteq\Omega\setminus \Omega_0$. Observe that $\mu(\Omega_0)>0$, otherwise  we would have $\mu(B)=\mu(B\cap(\Omega\setminus\Omega_0)) \leq 0$ which is a contradiction with the previously established fact that $\mu(B)>0$. This implies $\Omega_0\in\Gcal\setminus \Ncal_\Gcal$: indeed, if this were not the case, $T(f\ind_{\Omega_0})=0=T((g+\varepsilon)\ind_{\Omega_0})$ by Lemma \ref{lemma:pasting} item (ii) and definition of $\Ncal_\Gcal$, which in turns gives by Proposition \ref{V[A]orderpres} $\mu(\Omega_0)=V_{\Omega_0}[f]-V_{\Omega_0}(g+\varepsilon)\leq 0$, a contradiction.

\end{proof}

\begin{corollary}
\label{cor:aumentare una g}
Suppose Assumption \ref{ass:conditional} and Assumption \ref{ass:conditional:local} are satisfied and $\Pi(\Gcal)\neq\emptyset$.  Suppose that  $g\in\brvG$ satisfies $T(f\ind_A)\geq T(g\ind_A)$ for every $A\in\Gcal$ and the inequality is strict for some $ B\in\Gcal$. Then there exists a set $\Omega_0\in\Gcal\setminus \Ncal_\Gcal$ and $\varepsilon>0$ such that $T(f\ind_A)\geq T\left((g+\varepsilon\ind_{\Omega_0})\ind_A\right)$ for all $A\in\Gcal$.
\end{corollary}
\begin{proof}
Take $\varepsilon>0$, $\Omega_0\in\Gcal\setminus \Ncal_\Gcal$ as in Proposition \ref{prop: increase locally}. Take $A\in\Gcal$. We then have
\begin{align*}
(\star):&\quad T(f\ind_{A\cap (\Omega\setminus \Omega_0)})\geq T(g\ind_{A\cap (\Omega\setminus \Omega_0)})=T((g+\varepsilon \ind_{\Omega_0})\ind_{A\cap (\Omega\setminus \Omega_0)})\text{ by assumption }\\
(\star\star):&\quad T(f\ind_{A\cap \Omega_0})\geq T((g+\varepsilon \ind_{\Omega_0})\ind_{A\cap \Omega_0})\text{ by Proposition }\ref{prop: increase locally}.    
\end{align*}

Now applying Lemma \ref{lemma:pasting}  we have
\begin{align*}
T(f\ind_A)&=T(f\ind_{A\cap\Omega_0}+f\ind_{A\cap (\Omega\setminus \Omega_0)})
\\ &\stackrel{}{\geq } T((g+\varepsilon \ind_{\Omega_0})\ind_{A\cap\Omega_0}+  (g+\varepsilon \ind_{\Omega_0})\ind_{A\cap (\Omega\setminus \Omega_0)})=T\left((g+\varepsilon\ind_{\Omega_0})\ind_A\right)
\end{align*}

\end{proof}

We conclude this section with  proof of the main result of the paper, namely Theorem \ref{thm:existencecond NEW}.

\begin{proof}[Proof of Theorem \ref{thm:existencecond NEW}]Since the case $\Pi(\Gcal)=\emptyset$ is already covered by Lemma \ref{lemma: thmfor pig empty}, we consider now the case $\Pi(\Gcal)\neq\emptyset$. Observe that Assumption \ref{ass:conditional} and Assumption \ref{ass:conditional:local} are satisfied, so that we can exploit all the tools we have previously developed.
Recall that $A\mapsto V_A(\ind_{\Omega})$ is a probability measure on $\Gcal$, call it $\PW$ (see Theorem \ref{thm: reprTwithmeasure} item 1). Observe that for $g_1,g_2\in\brvG$, $g_1=g_2\,\PW-$a.s. implies $T(g_1)=T(g_2)$ (see Remark \ref{rem: equiv relation}). We introduce the set $$\Tcal_\Gcal:=\{g\in\brvG\mid T(f\ind_A)\geq T(g\ind_A)\,\forall\,A\in\Gcal\}.$$
Observe that $\Tcal_\Gcal\neq \emptyset$ since $-\snorm{f}\in\Tcal_\Gcal$. Furthermore, whenever $g_1,g_2\in \Tcal_\Gcal$ and $A\in\Gcal$, also $g_1\ind_A+g_2\ind_{\Omega\setminus A}\in\Tcal_\Gcal$: indeed $\Gcal-$ measurability is easily seen, and for any $B\in\Gcal$ we have by Lemma \ref{lemma:pasting}
\begin{align*}
&T((g_1\ind_A+g_2\ind_{\Omega\setminus A})\ind_B)=T(g_1\ind_{A\cap B}+g_2\ind_{(\Omega\setminus A)\cap B}) \\&\leq  T(f\ind_{A\cap B}+f\ind_{(\Omega\setminus A)\cap B})=T(f\ind_B) 
\end{align*}
Hence, the set $\Tcal_\Gcal$ is upward directed meaning that whenever $g_1,g_2\in \Tcal_\Scal$ we get that the pointwise maximum  $g_1\vee g_2$ belongs to $\Tcal_\Gcal$ (selecting $A=\{g_1\geq g_2\}$). Take now $\essup_\PW \{[g]_{\PW}\mid g\in \Tcal_{\Gcal}\}$, and $\hg$ a $\Gcal$-measurable representative for the equivalence class of the essential supremum. There exists a maximizing sequence $(g_n)_n\subseteq \Tcal_\Gcal$ since $\Tcal_\Gcal$ is upward directed, for which have $g_n\uparrow_n \hg$ $\PW-$a.s. (See \cite[Section A.5]{FS11} ). We can actually assume that $g_n(\omega)\uparrow_n \hg(\omega)\,\forall\omega\in\Omega$: by Remark \ref{rem: equiv relation} we can modify each $g_n$ on a set of probability zero, setting it to be equal to $\hg$ itself there, without affecting the values of $T(g_n\ind_A),A\in\Gcal$. By ($\Gcal$-PC) we have $T(\hg\ind_A)\leq T(f\ind_A)$ for all $A\in\Gcal$, since such an inequality is satisfied by $(g_n)$ for each $n$. It is not hard to see that $-\snorm{f}\leq \hg\leq \snorm{f},\PW-$a.s.: first observe that $-\snorm{f}\in\Tcal_\Gcal$ as previously observed, hence by definition $\hg\geq -\snorm{f},\PW-$a.s.. If moreover we had $\PW(\{\hg(\omega)>\snorm{f}\})>0$, we would have that for $n$ big enough  $V_{\Omega}(\ind_{E_n}))=\PW(E_n)>0$ for and $E_n:=\{g_n>\snorm{f}+\frac{1}{n}\}$. This in turns would give $E_n\notin\Ncal_\Gcal$ (Theorem \ref{thm: reprTwithmeasure} item 1), and 
by ($\Gcal$-Mo) we would conclude $T(g_n\ind_{E_n})\geq T((\snorm{f}+\frac{1}{n})\ind_{E_n})>T(\snorm{f}\ind_{E_n})\geq T(f\ind_{E_n})$, contradicting $g_n\in\Tcal_\Gcal$. Hence we can assume that $\abs{\hg(\omega)}\leq \snorm{f}, \forall \omega\in\Omega$ again modifying each $g_n$ on a set of probability zero, without affecting the values of $T(g_n\ind_A),A\in\Gcal$,  by Remark \ref{rem: equiv relation}. We conclude that $\hg\in\Tcal_\Gcal$. If we had that for some $B\in\Gcal$, $T(f\ind_B)>T(\hg\ind_B)$ we would have by Corollary \ref{cor:aumentare una g} that $T(f\ind_A)\geq T\left((\hg+\varepsilon\ind_{\Omega_0})\ind_A)\right),\forall A\in\Gcal$ for some $\varepsilon>0, \Omega_0\in\Gcal$ with $\Omega_0\in\Gcal\setminus\Ncal_\Gcal$, which implies $\hg+\varepsilon\ind_{\Omega_0}\in\Tcal_\Gcal$. Since $\Omega_0\in\Gcal\setminus\Ncal_\Gcal$ we have $\PW(\Omega_0)>0$, so that $\PW(\{\hg+\varepsilon\ind_{\Omega_0}>\hg\})>0$. This contradicts the maximality of $\hg$ as a representative of the essential supremum. Finally, the claimed essential uniqueness follows from Lemma \ref{lemma:uniqueness}.

\end{proof}

\appendix

\renewcommand{\thesection}{\Alph{section}}
\section{Appendix}
\subsection{On Debreu's Representation}\label{debreu}
\begin{definition}
\label{defop}
Let  $\Xcal$ be a set, and let $\Psi:\Xcal\rightarrow\R$ be given.
A map $\Phi:\Xcal\rightarrow\R$ is $\Psi$-order preserving on $\Xcal$ if for any two functions $f,g\in\Xcal$ we have: $\Phi(f)\leq \Phi(g)\Leftrightarrow \Psi(f)\leq \Psi(g)$.
\end{definition}

\begin{lemma}
\label{lemmaop}
For a functional $\Phi:\Xcal\rightarrow\R$ the following are equivalent:
\begin{enumerate}[(i)]
    \item \label{opleq} $\Phi$ is $\Psi$-order preserving on on $\Xcal$;
    \item \label{opstrict} for any two functions $f,g\in\Xcal$ we have: $\Phi(f)< \Phi(g)\Leftrightarrow \Psi(f)< \Psi(g)$.
\end{enumerate} 
\end{lemma}
\begin{proof}$\,$

\eqref{opleq}$\Rightarrow$\eqref{opstrict}: For a $\Psi$-order preserving $\Phi$, $\Phi(f)< \Phi(g)$ but $\Psi(f)=\Psi(g)$ implies $\Psi(f)\geq \Psi(g)$ and $\Phi(f)\geq \Phi(g)$, a contradiction, and $\Psi(f)<\Psi(g)$ but $\Phi(f)= \Phi(g)$ implies $\Phi(f)\geq\Phi(g)$, which in turns yields $\Psi(f)\geq \Psi(g)$, a contradiction.

\eqref{opstrict}$\Rightarrow$\eqref{opleq}: if $\Psi(f)=\Psi(g)$ and $\Phi(f)\gtrless\Phi(g)$ then by \eqref{opstrict} $\Psi(f)\gtrless\Psi(g)$, a contradiction, and if $\Phi(f)=\Phi(g)$ but $\Psi(f)\gtrless \Psi(g)$ then $\Phi(f)\gtrless\Phi(g)$, a contradiction.

\end{proof}

\medskip

Let $N = \{ 1, ... , n \}$ be a set of indexes and fix a subset $I \subset N$ given by $I=\{i_1< i_2< \ldots < i_k\}$. Then the complementary indexes are given by $I^c=\{j_1<\ldots<j_{n-k}\}$ and $I\cup I^C= N$. Given two vectors $y\in\R^k$ and $z\in \R^{n-k}$ we define the aggregated vector $yIz\in\R^n$ given by the vector formed by $y$ and $z$ conserving the order of the indexes $I$. More precisely if we call $w=yIz$ then $w_{i_h}=y_h$ for any $h=1,\ldots,k$ and $w_{j_h}=z_h$ for every $h=1,\ldots,n-k$.    
\begin{definition}
 A given preorder $\succeq$ on $\R^n$ induces on $\R^{k}$ a preorder $\succeq_I$, within the arrangement $I = \{i_1, ... , i_k\}$ for $k<n$, by
\begin{equation*}
(x_{1}, ... , x_k) \succeq_{I} (y_{1}, ... , y_k) \Leftrightarrow 
\exists z\in\R^k  :
xIz \succeq yIz.
\end{equation*}
\end{definition}

\begin{definition}\label{Sure:Thing}
The preorder $\succeq$ on $\R^n$ satisfies the Sure Thing Principle if for any $I \subset N$, $\succeq_{I}$ is independent of the choice of $z$. Namely if there exists $z\in\R^k$ such that $xIz \succeq yIz$ then $xIz' \succeq yIz'$ for any other $z'\in \R^k$.
\end{definition}

\begin{definition}
An index $I=\{i\}$ is said to be irrelevant if for any $x\in\R^{n-1}$ we have $xIz\sim xIz'$ for any $z,z'\in\R$; otherwise it is said to be essential.
\end{definition}

The next result is the main theorem proved by Debreu in \cite{De60}.

\begin{theorem}\label{Debreu}
Let $\succeq$ be a complete preordering of $\R^n$ such that $\forall y \in \R^n$ the sets $\{ x \in \R^n : x \succeq y\}$ and $\{ x \in \R^n : y \succeq x\}$ are closed. If $\succeq$ satisfies the Sure Thing Principle and more than two indexes are essential, then there is a continuous utility function determined up to an increasing affine transformation\footnote{For any other order preserving, continuous $V:\R^n\rightarrow\R$ such that $V(x)=\sum_{i=1}^N V_i(x_i)$ where $V_i:\R\rightarrow\R$ for $i=1,\dots,N$, there exists real constants $a>0,b$ such that $U(x)=a V(x)+b$ for every $x\in\R^N$.}, i.e. there is $U : \R^n \to \R$ order preserving and continuous such that
\begin{equation}
U(x) = \sum_{i=1}^{n} U_i(x_i) \qquad \forall x=(x_1, ... , x_n) \in \R^n,
\end{equation}
for $U_i : \R \to \R$ $\forall i=1, ... , n$. 
\end{theorem}

\subsection{Auxiliary results}

\begin{lemma}\label{extension:brv NEW}
Let $\Ucal:\brv\rightarrow\R$ be given, such that $\Ucal$ satisfies ($\Gcal$-PC) and
\begin{equation}
\label{weak monot general}    
 g_1,g_2\in\brvG,g_1(\omega)\leq g_2(\omega)\,\forall\omega\in\Omega\Longrightarrow \Ucal(g_1)\leq \Ucal(g_2).
\end{equation}
Let $\hPhi:\Scal(\Gcal)\rightarrow \R$ be $\Ucal$-order preserving on $\Scal(\Gcal)$, with $x\mapsto \hPhi(x)$ continuous on $\R$. Then $\hPhi$ extends to a ($\Gcal$-PC), $\Ucal$-order preserving functional on $\brvG$. 
\end{lemma}
\begin{proof}
Observe that since $\hPhi$ is $\Ucal$-order preserving on $\Scal(\Gcal)$, we have $\hPhi(g_1)\leq \hPhi(g_2)$ whenever we take $g_1,g_2\in\Scal(\Gcal)$ with $g_1(\omega)\leq g_2(\omega)\,\forall\,\omega\in\Omega$.
Define 
$$\Phi(g):=\inf\{\hPhi(s)\mid s\in\Scal(\Gcal),s(\omega)\geq g(\omega)\,\forall\,\omega\in\Omega\},\quad g\in\brvG.$$  
Clearly $\hPhi(\norm{g}_\infty)\geq \Phi(g)\geq \hPhi(-\norm{g}_\infty)$, and the definition is consistent i.e. $\Phi(g)=\hPhi(g)$ for every $g\in\Scal(\Gcal)$. $\Phi$ is monotone as
\begin{equation}
    \label{weak monot Phi}
    g_1,g_2\in\brvG,g_1(\omega)\leq g_2(\omega)\,\forall\omega\in\Omega\Longrightarrow \Phi(g_1)\leq \Phi(g_2).
\end{equation}

Moreover we have the following property: let $\{h_n\}_n$ be a minimizing sequence for $\Phi(g)$ and $\{g_n\}_n\subset \Scal(\Gcal)$ is such that $g_n\geq f$, and $\norm{g_n-g}_\infty\leq \frac{1}{n}$, then $s_n=h_n\wedge g_n$ is still simple for each $n$, $\norm{s_n-f}\leq \frac1n\rightarrow_n 0$, $\Phi(f)\leq \hPhi(s_n)\leq \hPhi(h_n)\rightarrow_n \Phi(f)$ by monotonicity of $\hPhi$ on simple functions. 

We now show that $g\mapsto \Phi(g)$ is ($\Gcal$-PC) on $\brvG$. Take a norm bounded sequence $\{g_n\}_n\subseteq \brvG$ such that  $g_n\rightarrow_n g $ pointwise. Suppose that $\Phi(g_n)$ is not converging to $\Phi(g)$. Then up to taking a subsequence (relabelled again with $n$) we have $2\varepsilon\leq \inf_n\abs{\Phi(g_n)-\Phi(g)}$ for some $\varepsilon>0$.   

Observe, by what was argued at the beginning of this proof, that we can take sequences of simple functions $s_n,v_n$ with, for all $n$ big enough,
\begin{center}
  \begin{tabular}{  c c c c c}
 $s_n\geq g_n$  & (i); & $\,\,\,$ & $v_n\geq g$ & (ii);\\
 $\norm{s_n-g_n}_\infty\leq\frac1n$ & (iii);  & $\,\,\,$ &  $\norm{v_n-g}_\infty\leq \frac1n$ & (iv).
\end{tabular}  
\end{center}
and 
\begin{center}   
\begin{tabular}{ c c}
   $\abs{\Phi(s_n)-\Phi(g_n)}=\abs{\hPhi(s_n)-\Phi(g_n)}\leq\frac1n$& (v); \\ 
   $\abs{\Phi(v_n)-\Phi(g)}=\abs{\hPhi(v_n)-\Phi(g)}\leq\frac1n$ & (vi).   
\end{tabular} 
\end{center}
which yield the existence of $N$ such that $\varepsilon\leq \inf_{n>N}\abs{\Phi(s_n)-\Phi(v_n)}$. Now we have two alternatives: either (a) there is a subsequence (to be relabelled again with $n$) such that $\hPhi(s_n)=\Phi(s_n)\geq \varepsilon+\Phi(v_n)=\varepsilon+\hPhi(v_n)$ for all $n$, or (b) there is a subsequence (to be relabelled again with $n$) such that $\hPhi(s_n)=\Phi(s_n)\leq -\varepsilon+\Phi(v_n)=-\varepsilon+\hPhi(v_n)$  for all $n$.

In case (a)  observe that for $n$ big enough we have 
by (v) and (vi) that
\begin{center}   
\begin{tabular}{ c c c c c}
   $\hPhi(v_n)\leq \frac{\varepsilon}{2}+\Phi(g)$ & (vii); &\,\,\,&   $\Phi(g_n)+\frac{\varepsilon}{3}\geq \hPhi(s_n)$ & (viii).   
\end{tabular} 
\end{center}
 Then, since in case (a) it holds that $\hPhi(s_n)\geq \varepsilon+\hPhi(v_n)\geq \varepsilon+\Phi(g)$ (the last inequality coming from (ii) and the definition of $\Phi$),  we conclude  $\frac{\varepsilon}{3}+\Phi(g_n)\geq \hPhi(s_n)\geq \varepsilon+\Phi(g)$, so that
 \begin{center}   
\begin{tabular}{ c c }
   $\Phi(g_n)\geq \frac{2\varepsilon}{3}+\Phi(g)$ & (ix). 
\end{tabular} 
\end{center}
We infer that
\begin{align*}
  \sup_{x\in\R}\hPhi(x)&\geq \hPhi(s_n)\stackrel{(i)}{\geq} \Phi(g_n)\stackrel{(ix)}{\geq}\frac{2\varepsilon}{3}+ \Phi(g)\\
  &> \frac{\varepsilon}{2}+\Phi(g)\stackrel{(vii)}{\geq} \hPhi(v_n)\stackrel{(ii)}{\geq} \Phi(g)\geq \inf_{x\in\R}\hPhi(x).  
\end{align*}

Since $x\mapsto \hPhi(x)$ is norm continuous on $\R$ and $\hPhi(\R)$ is connected, there exist $x_1,x_2\in\R$ with 
$$\hPhi(s_n)\geq \Phi(g_n)\geq\frac{2\varepsilon}{3}+ \Phi(g)> \Phi(x_1)>\Phi(x_2)>\frac{\varepsilon}{2}+\Phi(g)\geq \hPhi(v_n)\geq \Phi(g)$$ and in particular since $\Phi(x_1)=\hPhi(x_1),\Phi(x_2)=\hPhi(x_2)$
$$\hPhi(s_n)> \hPhi(x_1)>\hPhi(x_2)> \hPhi(v_n)$$
We see that we must have $x_1>x_2$ by \eqref{weak monot Phi}.  Since for $s_1,s_2$ simple functions $\hPhi(s_1)< \hPhi(s_2)$ if and only if $\Ucal(s_1)<\Ucal(s_2)$ (see Lemma \ref{lemmaop}) we conclude that $\Ucal(s_n)>\Ucal(x_1)>\Ucal(x_2)>\Ucal(v_n)$. This is clearly a contradiction since $\Ucal$ satisfies ($\Gcal$-PC)  and we have  $v_n \rightarrow_n g$ and $s_n\rightarrow_n g$ pointwise (the second convergence following from (iii), (iv)).

In case (b) for $n$ big enough  $\hPhi(v_n)\leq \frac{\varepsilon}{2}+\Phi(g)$ and
$$ \inf_{x\in\R}\hPhi(x)<\hPhi(s_n)\leq-\varepsilon+ \hPhi(v_n)\leq -\frac{\varepsilon}{2}+\Phi(g)<  \Phi(g) \leq  \hPhi(s_n)<  \sup_{x\in\R}\hPhi(x).$$
Since $\hPhi(\R)$ is again connected for the same reason as in case (a), there exist $x_1,x_2\in\R$ with 
$$ \hPhi(s_n)\leq-\varepsilon+ \hPhi(v_n)\leq -\frac{\varepsilon}{2}+\Phi(g)<\Phi(x_1)<\Phi(x_2)<  \Phi(g) \leq  \hPhi(s_n) $$ and $x_1<x_2$ by \eqref{weak monot Phi}.  Using again Lemma \ref{lemmaop} we conclude that $\Ucal(s_n)<\Ucal(x_1)<\Ucal(x_2)<\Ucal(v_n)$. This is clearly contradicts again ($\Gcal$-PC).

We finally show that $\Phi$ is $\Ucal$-order preserving on $\brvG$. Let indeed $g_1,g_2\in\brvG$ be given. Then we can take sequences of simple functions $g^1_n\uparrow g_1$ and $g^2_n\downarrow g_2$.
If $\Ucal(g_1)\leq \Ucal(g_2)$, by \eqref{weak monot general} we have  $\Ucal(g^1_n)\leq \Ucal(g_1)\leq \Ucal(g_2)\leq \Ucal(g^2_n)$, thus $\Ucal(g^1_n)\leq \Ucal(g^2_n)$, and since $\hPhi$ is $\Ucal$-order preserving on $\Scal(\Gcal)$ we have $\Phi(g^1_n)=\hPhi(g^1_n)\leq \Phi(g^2_n)=\Phi(g_n^2)$. Taking the limit the inequality is preserved and, recalling that $\hPhi$ is ($\Gcal$-PC), we get $\Phi(g_1)\leq \Phi(g_2)$. If conversely $\Phi(g_1)\leq \Phi(g_2)$, by \eqref{weak monot Phi} we get $\Phi(g^1_n)\leq \Phi(g_1)\leq \Phi(g_2)\leq \Phi(g^2_n)$, thus $\hPhi(g^1_n)=\Phi(g^1_n)\leq \Phi(g^2_n)=\hPhi(g_n^2)$ and since $\hPhi$ is $\Ucal$-order preserving we get $\Ucal(g^1_n)\leq \Ucal(g_n^2)$ for all $n$. Passing to the limit and recalling that $\Ucal$ is ($\Gcal$-PC) we conclude $\Ucal(g_1)\leq \Ucal(g_2)$.

\end{proof}

\begin{lemma}
\label{lemma: aux1}
Assume that $T$ satisfies ($\Gcal$-QL), $T(0)=0$, and the property in Eq. \eqref{weak monot}.
Then property in Eq. \eqref{formulalternativeN} holds.
\end{lemma}
\begin{proof}
Clearly $\Ncal_\Gcal\subseteq \left\{A\in\Gcal\mid T(x\ind_A)=0\text{ for all }x\in\R\right\}$ by taking $g_1=0$ in Eq. \eqref{def:ncal initial}, while for the converse start  observing that given $g_1,g_2\in\brvG$, since for all $\omega\in\Omega$
\begin{align*}
  &g_1(\omega)\ind_{\Omega\setminus A}(\omega)-(\snorm{g_1}+\snorm{g_2})\ind_A(\omega)\\
  &\leq g_1(\omega)+g_2(\omega)\ind_{A}(\omega)\\
  &\leq g_1(\omega)\ind_{\Omega\setminus A}(\omega)+(\snorm{g_1}+\snorm{g_2})\ind_A(\omega),  
\end{align*}
then we have by Eq. \eqref{weak monot} that 
\begin{align*}
  &T(g_1\ind_{\Omega\setminus A}-(\snorm{g_1}+\snorm{g_2})\ind_A\\
  &\leq T(g_1+g_2\ind_{A})\\
  &\leq T(g_1\ind_{\Omega\setminus A}+(\snorm{g_1}+\snorm{g_2})\ind_A).  
\end{align*}
 If  $A\in\Gcal$ satisfies $T(x\ind_A)=0=T(0)$ for all $x\in\R$, by ($\Gcal$-QL) we get that  $T(g_1\ind_{\Omega\setminus A}\pm(\snorm{g_1}+\snorm{g_2})\ind_A)=T(g_1\ind_{\Omega\setminus A})$. We conclude showing that $T(g_1\ind_{\Omega\setminus A})=T(g_1)$ (which does not come automatically from ($\Gcal$-QL)) yielding  $T(g_1+g_2\ind_A)=T(g_1)$ and, in turns, $A\in\Ncal_\Gcal$ as defined in Eq. \eqref{def:ncal initial}. Observe indeed that, arguing as above, $T(g_1\ind_{\Omega\setminus A}-\snorm{g_1}\ind_A)\leq T(g_1)\leq T(g_1\ind_{\Omega\setminus A}-\snorm{g_1}\ind_A)$, the latter providing the desired equality via ($\Gcal$-QL).

\end{proof}

\begin{lemma}[Properties of irrelevant events]
\label{rem: propertiesirrelevant}
Suppose that $T$ satisfies $T(0)=0$ and ($\Gcal$-QL). Then:
\begin{enumerate}[(i).]
    \item if $A,B\in\Gcal$, $A\subseteq B$ and $B\in\Ncal_\Gcal$ then $A\in\Ncal_\Gcal$;
    \item whenever $\{A_n\}_n\subseteq \Ncal_\Gcal$ is a finite collection  of irrelevant events, their union is irrelevant as well (i.e. $\bigcup_n A_n\in\Ncal_\Gcal$). 
\end{enumerate}
If additionally $T$ satisfies ($\Gcal$-PC), item (ii) holds also for countable collections 
\end{lemma}
\begin{proof}
(i) follows observing that $T(g_1+g_2\ind_A)=T(g_1+g_2\ind_A\ind_B)=T(g_1)$ using in the last equality that $g_2\ind_A\in\brvG$ and $B\in\Ncal_\Gcal$. 
We check (ii) for the countable case $\{A_n\}_{n\in\NN}$, since the finite one is simpler but similar. Without loss of generality we can assume the sets $A_n$ of being pairwise disjoint, otherwise we might reduce to this setup taking suitable subsets (which will still belong to $\Ncal_\Gcal$ by (i)). For any $g_1,g_2\in\brvG$ we have for any $g_1,g_2\in\brvG$ \[T(g_1)=T(g_1+g_2\ind_{A_1})=\dots=T\left(g_1+g_2\sum_{j=1}^N\ind_{A_j}\right)=T\left(g_1+g_2\ind_{\bigcup_{j=1}^NA_j}\right)\]
by iteration of the defining property in Eq. \eqref{def:ncal initial}. Thus $\bigcup_{j=1}^NA_j\in \Ncal_{\Gcal}$. Taking the limit as $N\rightarrow\infty$ and using ($\Gcal$-PC) we get analogously that $\bigcup_nA_n\in\Ncal_\Gcal$ (observe that the pointwise limit of $g_1+g_2\ind_{\bigcup_{j=1}^NA_j}$ as $N\rightarrow \infty$ is in fact $g_1+g_2\ind_{\bigcup_{n}A_n}$).

\end{proof}

\begin{lemma}[On monotonicity]
\label{remonmonot}
Assume $T:\mathcal{L}^{\infty}(\Omega,\Fcal)\to \R$ satisfies $T(0)=0$, ($\Gcal$-Mo), ($\Gcal$-QL), and that the restriction of $T$ on $\brvG$ is $\snorm{\cdot}$ continuous (which is implied by ($\Gcal$-PC)). Then $T$ is monotone on $\brvG$, in that Eq. \eqref{weak monot} holds. Moreover, the monotonicity is strict: if $g_1,g_2\in\brvG$ satisfy $g_1(\omega)\leq g_2(\omega)\,\forall\omega\in\Omega$ and $\{g_1<g_2\}\in\Gcal\setminus\Ncal_\Gcal$, then $T(g_1)<T(g_2)$.
\end{lemma}

\begin{proof}
Consider $g_1,g_2\in \brvG$, $g_1(\omega)\leq g_2(\omega)\,\forall\omega\in\Omega$. We consider the set $A=\{g_1<g_2\}$ and $\Omega\setminus A=\{g_1=g_2\}$. If  $A\in \Ncal_\Gcal$, $T(g_1)= T(g_1\ind_A+g_1\ind_{\Omega\setminus A})=T(g_1\ind_A+g_2\ind_{\Omega\setminus A})=T(g_2)$.

 We now suppose $A\in\Gcal\setminus \Ncal_\Gcal$. We prove that $T(g_1)\leq T(g_2)$.
 For the moment, let us suppose that $g_1,g_2\in\Scal(\Gcal)$ are simple functions, with $g_1(\omega)\leq g_2(\omega)$ for all $\omega\in\Omega$ and $\{g_1<g_2\}\in\Gcal\setminus\Ncal_\Gcal$. We show that $T(g_1)\leq T(g_2)$.
 We can write $g_1=\sum_{A\in\pi}x^1_A\ind_A$, $g_2=\sum_{A\in\pi}x^2_A\ind_A$ for a (common) partition $\pi\subseteq \Gcal$. Then $\{g_1<g_2\}=\bigcup\{A\in\pi\mid x_A^1<x_A^2\}$. By Lemma \ref{rem: propertiesirrelevant} we must have that at least one of the sets among the $\{A\in\pi\mid x_A^1<x_A^2\}$ needs not belong to $\Ncal_\Gcal$. Call $\pi^+:=\{A\in\pi\mid A\in\Gcal\setminus \Ncal_\Gcal, x^1_A<x^2_A\}\neq \emptyset$, $\pi^=:=\{B\in\pi\mid B\in\Gcal\setminus \Ncal_\Gcal, x^1_B=x^2_B\}(\neq \emptyset)$  and $\pi^0:=\{B\in\pi\mid B\in\Ncal_\Gcal\}$. Clearly $\pi=\pi^+\cup\pi^=\cup\pi^0$. To simplify the discussion, in the equations below in case either $\pi^==\emptyset$ or $\pi^0=\emptyset$ the corresponding summation is set to $0$ by defult, so that it is simply ignored.  Setting $\bar C=\bigcup_{C\in \pi^0}C$, we see
 \small{
\begin{align}
T(g_1)&=T\left(\sum_{A\in\pi^+}x^1_A\ind_A+\sum_{B\in\pi^=}x^1_B\ind_B+\sum_{C\in\pi^0}x^1_C\ind_C\right)  \label{1-2}\\
&<T\left(x^2_{A_0}\ind_{A_0}+\sum_{A\in\pi^+, A\neq A_0}x^1_A\ind_A+\sum_{B\in\pi^=}x^1_B\ind_B+\sum_{C\in\pi^0}x^1_C\ind_C\right) 
\label{2-3}\\
&<T\left(\sum_{A\in\pi^+}x^2_A\ind_A+\sum_{B\in\pi^=}x^1_B\ind_B+\left(\sum_{C\in\pi^0}x^1_C\ind_C\right)\ind_{\bar C} \right)\label{3-4}\\
&=T\left(\sum_{A\in\pi^+}x^2_A\ind_A+\sum_{B\in\pi^=}x^2_B\ind_B+\left(\sum_{C\in\pi^0}x^2_C\ind_C\right)\ind_{\bar C}\right)\label{5-6}\\
&=T\left(\sum_{A\in\pi^+}x^2_A\ind_A+\sum_{B\in\pi^=}x^2_B\ind_B+\sum_{C\in\pi^0}x^2_C\ind_C\right)=T(g_2).\notag
\end{align}}%
where between \eqref{1-2} and\eqref{2-3} we selected $A_0\in\pi^+$ and used ($\Gcal$-Mo), and the same procedure was iterated to get to \eqref{3-4}. Between \eqref{3-4} and \eqref{5-6} we used the fact that $x^1_B=x^2_B$ for every $B\in\pi^=$, and $\bar C\in \Ncal_{\Gcal}$.
In particular $T(g_1)<T(g_2)$. Suppose now $g_1,g_2\in\brvG$ are generic, as before. Consider sequences $\{g_1^n\}_{n\in\NN},\{g_2^n\}_{n\in\NN}$ such that $g_1^n\uparrow_ng_1, g_2^n\downarrow_ng_2$ with $\snorm{g^n_1-g_1}\downarrow_n 0$,  $\snorm{g^n_2-g_2}\downarrow_n 0$. Observe that then by construction for every $n\in\NN$ we have $\{g_1<g_2\}\subseteq \{g_1^n<g_2^n\}$, so that  $\{g_1^n<g_2^n\}\in\Gcal\setminus\Ncal_\Gcal$ for every $n\in\NN$ (if this were not the case, we would conclude that $\{g_1<g_2\}\in\Ncal_\Gcal$ by Lemma \ref{rem: propertiesirrelevant}). By the previous argument we get $T(g_1^n)< T(g_2^n)$ for all $n\in \NN$. Passing to the limit we get the desired inequality.
As to strict monotonicity, observe that if $\{g_1<g_2\}\in\Gcal\setminus\Ncal_\Gcal$, we must have for some $N\in\NN$  that $A_N:=\{g_1+\frac1N<g_2\}\in\Gcal\setminus\Ncal_\Gcal$ (since $\{g_1<g_2\}=\bigcup_N\{g_1+\frac1N<g_2\}$, if all the terms in RHS were irrelevant, so would be LHS by Lemma \ref{rem: propertiesirrelevant}). Thus denoting $\sup_{E}h:=\sup_{\omega\in E}h(\omega)$ and analogously for $\inf_{E}h$
\begin{align*}
T(g_1)&=T(g_1\ind_{A_N}+g_1\ind_{\Omega\setminus A_N})\leq T\left(\sup_{A_N}(g_1)\ind_{A_N}+g_1\ind_{\Omega\setminus A_N}\right)\\&<T\left(\left(\sup_{A_N}(g_1)+\frac{1}{N}\right)\ind_{A_N}+g_1\ind_{\Omega\setminus A_N}\right)\\
&\leq T\left(\inf_{A_N}(g_2)\ind_{A_N}+g_1\ind_{\Omega\setminus A_N}\right)\leq T(g_2)    
\end{align*}
where the strict inequality comes from ($\Gcal$-Mo), and all the other inequalities follow from Eq. \eqref{weak monot} that was just proved above (observe that in particular $\inf_{A_N}(g_2)\ind_{A_N}+g_1\ind_{\Omega\setminus A_N}\leq g_2$ on $\Omega$).

\end{proof}

\subsection{Proofs of Section \ref{example:RM}}\label{scal:RM}
In this section we provide all the mathematical details of the announced Theorem \ref{scal:RM:prop}. Indeed one must be careful when shifting from pointwise defined random variables in $\brv$ to equivalences classes in $\Linfty$.  
In particular as anticipated in Section \ref{preliminaries} we will use the following suggestive notation: capital letters stand for equivalence classes, and lower case letters stand for measurable functions. Typically for $ f\in \brv, g\in\brvG $ $X\in\Linfty$ and $Y\in\LinftyG$ we will write $f\in X, g\in Y$ meaning that $X=[f]_{\PW}$ is the equivalence class of $f$, and $Y=[g]_{\PW}$ is the equivalence class of $g$. We recall that for any $A\in \Fcal$ we shall denote by $\ind_A$ the indicator function in $\brv$ and by $\Ind_A=[\ind_A]_{\PW}$ the equivalence class generated in $\Linfty$.

\begin{proof}[Proof of Theorem \ref{scal:RM:prop}]

We start proving the reverse implication.  Let $$\rho_\Gcal:\Linfty\rightarrow\LinftyG$$ be a conditional convex Risk Measure with $\rho_\Gcal(0)=0$ $\PW$-a.s..  Then 
for every $X\in\Linfty$, $Y\in \LinftyG$ we have $\rho_0(X+Y)=\Ep{\rho_{\Gcal}(X+Y)}=\rho_0(X)-\Ep{Y}$. For $X_1\leq X_2$ $\PW$-a.s. we immediately have  $\Ep{ \rho_{\Gcal}(X_1)}\geq \Ep{ \rho_{\Gcal}(X_2)}$ $\PW$-a.s.. 
\\ We recall that for any conditional convex Risk Measure null in $0$ we have the pasting property $\rho_{\Gcal}(X_1)\Ind_{A}+\rho_{\Gcal}(X_2)\Ind_{\Omega\setminus A}=\rho_{\Gcal}(X_1\Ind_{A}+X_2\Ind_{\Omega\setminus A})$ $\PW$-a.s. for any $X_1,X_2\in \Linfty$ and $A\in\Gcal$, which implies the property $\rho_{\Gcal}(X)\Ind_{A}=\rho_{\Gcal}(X\Ind_{A})$ whenever $\rho_{\Gcal}(0)=0$ (see \cite[Proposition 1]{DetlefsenScandolo05}  for further details). In particular, $\rho_0$ automatically satisfies ($\Gcal$-PS). 
\\Let  $N\geq 1$, $0\leq \lambda_1,\dots,\lambda_N\leq 1$ and $A_1,\dots,A_N\in\Gcal$ a partition of $\Omega$.  We have for all $X_1,X_2\in \Linfty$ and $\Lambda=\sum_{j=1}^N\lambda_j\Ind_{A_j}$
\begin{align*}
\rho_0\left(\Lambda X_1+\left(1-\Lambda\right)X_2\right)& =  \Ep{\rho_{\Gcal}(\Lambda X_1+\left(1-\Lambda\right))X_2} 
\\ & \leq  \Ep{\Lambda\rho_{\Gcal}( X_1)+\left(1-\Lambda\right)\rho_{\Gcal}(X_2)} 
\\ & =  \sum_{j=1}^N\lambda_j \rho_0(X_1\Ind_{A_j})+\sum_{j=1}^N(1-\lambda_j) \rho_0(X_2\Ind_{A_j}),  
\end{align*}
where in the last equality we applied the property that $\rho_{\Gcal}(X)\Ind_{A}=\rho_{\Gcal}(X\Ind_{A})$ for any $X\in \Linfty$ and $A\in\Gcal$. 

We now show the direct implication. Let $\rho_0:\Linfty\to \R$ satisfy $\rho_0(0)=0$ jointly with conditions in items 1 to 4 of Theorem \ref{scal:RM:prop}. We define the new functional $T:\brv\to\R$ as the map 
$$f\mapsto T(f):= \rho_0(-[f]_{\PW}).$$
The functional $T$ inherits from $\rho_0$ the following properties: $T(0)=0$, and 
$$f_1,f_2\in\brv,f_1(\omega)\leq f_2(\omega)\,\forall\omega\in\Omega\Longrightarrow T(f_1)\leq T(f_2),$$
$$T(f+g)=T(f)+\Ep{g} \quad \forall\; f\in\brv\;g\in \brvG.$$

\noindent $T$ is also ($\Gcal$-QL) since it is linear of $\LinftyG$ (by $\rho_0(0)=0$ and the property in item 1 of Theorem \ref{scal:RM:prop}).
We see then (by \eqref{formulalternativeN} and the discussion leading to it)
$$\mathcal{N}_{\Gcal}=\{A\in \Gcal\mid T(x\ind_A)=0\,\forall x\in\R\}=\{A\in\Gcal\mid \PW(A)=0\}.$$
Moreover $T$ is ($\Gcal$-Mo): for $x,y\in\R$ with $x<y$, all $g\in\brvG$, $A\in\Gcal\setminus\Ncal_\Gcal$ we have $$T(x\ind_A+g\ind_{\Omega\setminus A})=\Ep{x\ind_A+g\ind_{\Omega\setminus A}}< \Ep{y\ind_A+g\ind_{\Omega\setminus A}}=T(y\ind_A+g\ind_{\Omega\setminus A}).$$ 

\noindent Property ($\Gcal$-PC) for $T$ follows immediately from Dominated Convergence Theorem and $T(g)=\Ep{g}$ for any $g\in\brvG$.

\noindent Finally $T$ is ($\Gcal$-PS) by assumption, and ($\Gcal$-NB) as for every $A\in\Gcal$,  for $f\in\brv$,  $-\snorm{f}\ind_A \leq f\ind_A\leq \snorm{f}\ind_A$. This implies by monotonicity above that $T(-\snorm{f}\ind_A)\leq T(f\ind_A)\leq T(\snorm{f}\ind_A)$.

\medskip

The functional $T$ satisfies all the assumption of Theorem \ref{thm:existencecond NEW} and for any $f\in\brv$ we can guarantee the existence of the conditional Chisini mean $\cm{f\middle |\Gcal}$ such that for every $g\in \cm{f\middle |\Gcal}$ we have $T(f\ind_A)=T(g\ind_A)$ for all $A\in\Gcal$.

We now define $\rho_{\Gcal}:\Linfty \to \LinftyG$ by
$$\rho_{\Gcal}(X) := -[g]_{\PW} \quad \text{ where } g\in \cm{f\middle |\Gcal}.$$
First we observe that $\rho_{\Gcal}$ is well defined. On the one hand the definition does not depend on the choice of $g$ as $\PW(g=\tilde g)=1$ for all $g,\tilde g\in \cm{f\middle |\Gcal}$. On the one hand for any $f_1,f_2\in\brv$ such that $\PW(f_1=f_2)$ we have $\cm{f_1\middle |\Gcal} \equiv \cm{f_2\middle |\Gcal}$. Indeed, if by contradiction there existed $g_1\in \cm{f_1\middle |\Gcal}$ and  $g_2\in \cm{f_2\middle |\Gcal}$ such that $\Gcal\setminus \Ncal_\Gcal \ni A=\{g_1> g_2\}$  then $T(f_1\ind_A)=T(g_1\ind_A)=\Ep{g_1\ind_A}> \Ep{g_2\ind_A}= T(g_2\ind_A)= T(f_2\ind_A)$. But $T(f_1\ind_A)=\rho_0(-[f_1\ind_A]_{\PW})= \rho_0(-[f_2\ind_A]_{\PW})= T(f_2\ind_A)$, a contradiction. The case $\{g_1<g_2\}\in\Gcal\setminus \Ncal_\Gcal$ can be excluded in a similar way. On the other hand the definition does not depend on the choice of $g$ as $\PW(g=\tilde g)=1$ for all $g,\tilde g\in \cm{f\middle |\Gcal}$: if this were the case, one could reach a contradiction in the same way as we just did above.

We conclude this first part of our proof by showing that $\rho_{\Gcal}$ is a conditional convex Risk Measure. Before starting with the main properties we notice that 
\begin{equation}
    \label{eq locality rhoG}
\rho_{\Gcal}(X\Ind_A)=\rho_{\Gcal}(X)\Ind_A\,\,\,\PW\text{-a.s} \,\,\,\forall \,X\in\Linfty,\,A\in\Gcal,
\end{equation}
which can be checked by direct verification.
Let $X_1\leq X_2$ $\PW$-a.s. then we find $X_1\ni f_1\leq f_2 \in X_2$ and $g_1\in\cm{f_1|\Gcal},g_2\in\cm{f_2|\Gcal}$ so that for all $A\in\Gcal$ $$\Ep{g_1\ind_A}=T(f_1\ind_A)=\rho_0(-X_1\Ind_A)\leq \rho_0(-X_2\Ind_A)=T(f_2\ind_A)=\Ep{g_2\ind_A}.$$ This implies $\PW(g_1\leq g_2)=1$ i.e. 
$\rho_{\Gcal}(X_1)\geq \rho_{\Gcal}(X_2)$ $\PW$-a.s.. Let now $X\in\Linfty$ and $c\in\LinftyG$. For any $A\in\Gcal$  
$$\Ep{\rho_{\Gcal}(X+c)\Ind_A}=\rho_{0}((X+c)\Ind_A)= \rho_{0}(X\Ind_A)-\Ep{c\Ind_A}=\Ep{(\rho_{\Gcal}(X)-c)\Ind_A},$$
hence $\rho_{\Gcal}(X+c)=\rho_{\Gcal}(X)-c$ $\PW$-a.s..
Notice that this property, as usual when dealing with convex Risk Measures, yields Lipschitz continuity for $\rho_\Gcal$: if $X_1,X_2\in\Linfty$ are given, we have $$\rho_\Gcal(X_1)=\rho_\Gcal(X_1-X_2+X_2)\leq \rho_\Gcal(X_2-\snorm{X_1-X_2})=\rho_\Gcal(X_2)+\snorm{X_1-X_2}$$
and interchanging the roles of $X_1,X_2$ we get $$\snorm{\rho_\Gcal(X_1)-\rho_\Gcal(X_2)}\leq \snorm{X_1-X_2}.$$

Let $X_1,X_2\in \Linfty$ and $\Lambda\in \LinftyG$ with $0\leq \Lambda\leq 1$. We start assuming that $\Lambda$ is simple, in that  $\Lambda=\sum_{j=1}^N\lambda_j\Ind_{A_j}$ $\PW$-a.s. for $0\leq \lambda_1,\dots,\lambda_N\leq 1$ and $A_1,\dots,A_N\in\Gcal$ being a partition of $\Omega$.
\\ Suppose that for some $\varepsilon>0$ we had $\PW(B_\varepsilon)>0$ where \[B_\varepsilon=\left\{\rho_{\Gcal}(\Lambda X_1+(1-\Lambda)X_2)>\Lambda\rho_{\Gcal}(X_1)+(1-\Lambda)\rho_{\Gcal}(X_2)+\varepsilon\right\}\in\Gcal.\]
Then 
\begin{eqnarray*}
&&\rho_0(\Lambda (X_1\Ind_{B_\varepsilon})+(1-\Lambda)(X_2\Ind_{B_\varepsilon}))=\Ep{\rho_{\Gcal}(\Lambda (X_1\Ind_{B_\varepsilon})+(1-\Lambda)(X_2\Ind_{B_\varepsilon}))}\\
&=&\Ep{\rho_{\Gcal}(\Lambda (X_1)+(1-\Lambda)(X_2))\Ind_{B_\varepsilon}}\\
&>& \Ep{\Lambda\rho_{\Gcal} (X_1)\Ind_{B_\varepsilon}+(1-\Lambda)\rho_{\Gcal}(X_2)\Ind_{B_\varepsilon}}
\\& = & \Ep{\Lambda\rho_{\Gcal} (X_1\Ind_{B_\varepsilon})+(1-\Lambda)\rho_{\Gcal}(X_2\Ind_{B_\varepsilon})}
\\ & = &\Ep{\sum_{j=1}^N\lambda_j\Ind_{A_j}\rho_{\Gcal} (X_1\Ind_{B_\varepsilon})+\big(1-\sum_{j=1}^N\lambda_j\Ind_{A_j}\big)\rho_{\Gcal} (X_2\Ind_{B_\varepsilon})}
\\ & = & \sum_{j=1}^N\lambda_j\rho_{0} (X_1\Ind_{B_\varepsilon}\Ind_{A_j})+\sum_{j=1}^N\big(1-\lambda_j\big)\rho_{0} (X_2\Ind_{B_\varepsilon}\Ind_{A_j}) 
\end{eqnarray*}
using \eqref{eq locality rhoG} in the last step.
This is a contradiction, since in \eqref{eq:for convexity} we can choose $X\ind_{B_\varepsilon}$ and $Y\ind_{B_\varepsilon}$. We conclude that
\[\PW\left(\left\{\rho_{\Gcal}(\Lambda X_1+(1-\Lambda)X_2)\leq \Lambda\rho_{\Gcal}(X_1)+(1-\Lambda)\rho_{\Gcal}(X_2)+\varepsilon\right\}\right)=1\]
for every $\varepsilon>0$ which in turns implies $$\rho_{\Gcal}(\Lambda X_1+(1-\Lambda)X_2)\leq \Lambda\rho_{\Gcal}(X_1)+(1-\Lambda)\rho_{\Gcal}(X_2)\quad \PW\text{-a.s. } $$ 
for every $X_1,X_2\in \Linfty$ and simple $\Lambda\in \LinftyG$ with $0\leq \Lambda\leq 1$ $\PW$-a.s..  We can finally remove the requirement for  $\Lambda$ to be simple. Indeed we can find $\Lambda_n\in \LinftyG$, simple and with $0\leq \Lambda_n\leq 1$ $\PW$-a.s., satisfying $\snorm{\Lambda_n-\Lambda}\to 0$. Then 
\begin{align*}
   &\snorm{\Lambda_n X_1+(1-\Lambda_n)X_2-(\Lambda X_1+(1-\Lambda)X_2)}\to 0, \\
   &\snorm{\Lambda_n \rho_\Gcal(X_1)+(1-\Lambda_n)\rho_\Gcal(X_2)-(\Lambda \rho_\Gcal(X_1)+(1-\Lambda)\rho_\Gcal(X_2))}\to 0.
\end{align*}
By Lipschitz continuity of $\rho_\Gcal$ 
\begin{equation}\label{limit:A}\rho_\Gcal((\Lambda_n X_1+(1-\Lambda_n)X_2))\to \rho_\Gcal((\Lambda X_1+(1-\Lambda)X_2))\end{equation}
and  since we have proved that for every $n$
$$\rho_{\Gcal}(\Lambda_n X_1+(1-\Lambda_n)X_2)\leq \Lambda_n\rho_{\Gcal}(X_1)+(1-\Lambda_n)\rho_{\Gcal}(X_2) \to \Lambda \rho_\Gcal(X_1)+(1-\Lambda)\rho_\Gcal(X_2)$$
by \eqref{limit:A} we have the desired property.                                                                                                                       

\end{proof}

\bibliographystyle{abbrv} 
 \bibliography{cas-refs.bib}

\end{document}